\documentclass[10pt]{amsart}
\pdfoutput=1
\usepackage{microtype}
\usepackage[utf8]{inputenc}
\usepackage{amsthm}
\usepackage{amsmath}
\usepackage{amssymb}
\usepackage{amscd}
\usepackage{mathtools}
\usepackage{hyperref}
\usepackage{mathrsfs}

\usepackage [backend= biber, style= alphabetic, sorting= nty]{biblatex}
\hypersetup{
    colorlinks=true,
    linkcolor=cyan,
    filecolor=magenta,
    citecolor= cyan,
    urlcolor=blue}
    
\usepackage{cleveref}
\usepackage{relsize}
\usepackage{tikz-cd}
\usepackage[left=2.5cm,top=2.5cm,bottom=3cm,right=2.5cm]{geometry}

\newtheorem{theorem}{Theorem}[section]
\newtheorem{lemma}[theorem]{Lemma}
\newtheorem{proposition}[theorem]{Proposition}

\newtheorem{corollary}[theorem]{Corollary}
\theoremstyle{definition}
\newtheorem{notation}[theorem]{Notation and Convention}
\newtheorem{definition}[theorem]{Definition}
\newtheorem{remark}[theorem]{Remark}
\newtheorem{question}[theorem]{Question}

\newtheorem{example}[theorem]{Example}
\newtheorem*{theorem*}{Theorem}
\newtheorem*{claim}{Claim}

\newcommand{\expand}[2]{(\frac{1}{p^n})^{#2}\sum \limits_{j= -\infty}^{\infty}\lambda \left(({#1})_j\right) \ e^{-iyj/p^n}}
\newcommand{\Z}{\mathbb{Z}}
\newcommand{\N}{\mathbb{N}}
\newcommand{\bp}[2]{{#1}^{[{p}^{#2}]}}
\newcommand{\E}{\mathcal{E}}
\newcommand{\F}{\mathcal{F}}

\title{Frobenius-Poincar\'e function and Hilbert-Kunz multiplicity}
\author{Alapan Mukhopadhyay}
\date{}

\begin{document}
\maketitle
\begin{abstract}
    We generalize the notion of Hilbert-Kunz multiplicity of a graded triple $(M,R,I)$ in characteristic $p>0$  by proving that for any complex number $y$, the limit
    $$\underset{n \to \infty}{\lim}(\frac{1}{p^n})^{\text{dim}(M)}\sum \limits_{j= -\infty}^{\infty}\lambda \left( (\frac{M}{I^{[p^n]}M})_j\right)e^{-iyj/p^n}$$
    exists. We prove that the limiting function in the complex variable $y$ is entire and name this function the \textit{Frobenius-Poincar\'e function}. We establish various properties of Frobenius-Poincar\'e functions including its relation with the tight closure of the defining ideal $I$; and relate the study Frobenius-Poincar\'e functions to the behaviour of graded Betti numbers of $\frac{R}{I^{[p^n]}} $ as $n$ varies. Our description of Frobenius-Poincar\'e functions in dimension one and two and other examples raises questions on the structure of Frobenius-Poincar\'e functions in general.  
\end{abstract}
\footnotetext[0]{Mathematics Subject Classification 2020: Primary 13D40;  Secondary: 13A02, 13A35, 30D15 .}

\section{Introduction}

In this article, we introduce the \textit{Frobenius-Poincar\'e function} of a graded pair $(R,I)$, where $I$ is a finite co-length homogeneous ideal in the standard graded domain $R$ over a perfect field of positive characteristic $p$.  This function is holomorphic everywhere on the complex plane and is roughly the limit of the Hilbert series of the graded $R$-modules  $\frac{R^{1/p^n}}{IR^{1/p^n}}$ as $n$ goes to infinity. The Frobenius-Poincar\'e function encodes the information of the Hilbert-Kunz multiplicity of the pair $(R, I)$ along with other asymptotic invariants of $(R,I)$.\\

To be precise, fix a pair $(R, I)$ as above. For each positive integer $n$, consider the $R$-module $R^{1/p^n}$, the collection of $p^n$-th roots of elements of $R$ in a fixed algebraic closure of the fraction field of $R$. There is a natural $\frac{1}{p^n}\mathbb{Z}$-grading on $R^{1/p^n}$. So one can consider the  Hilbert series of $\frac{R^{1/p^n}}{IR^{1/p^n}}$ by allowing for rational powers of the variable $t$--- namely $\underset{\nu \in \frac{1}{p^n}\Z}{\sum} \lambda((\frac{R^{1/p^n}}{IR^{1/p^n}})_{\nu})t^{\nu}$. To study these Hilbert series  as \textit{holomorphic functions} on the complex plane, a natural approach is to replace $t$ by $e^{-iy}$ \footnote{Here $e$ is the complex number $\sum \limits_{j=0}^{\infty}\frac{1}{j!}$ and $i$ is a complex square root of $-1$, fixed throughout this article.}, which facilitates taking $p^n$-th roots as holomorphic functions. The process described above gives a sequence $(G_n)_n$ of holomorphic functions where $$G_n(y)= \underset{\nu \in \frac{1}{p^n}\Z}{\sum} \lambda((\frac{R^{1/p^n}}{IR^{1/p^n}})_{\nu})e^{-iy\nu}.$$ 
 Our main result, \Cref{the main theorem}, guarantees that the sequence of functions
$$\frac{G_n(y)}{(p^{\text{dim}(R)})^n }$$
converges, as $n$ goes to infinity, to a function $F(y)$ which is holomorphic everywhere on the complex plane. Furthermore, this convergence is uniform on every compact subset. We call this function $F(y)$ the \textit{Frobenius-Poincar\'e function} associated to the pair $(R,I)$.\\

 The Frobenius-Poincar\'e function can be viewed as a natural refinement of the Hilbert-Kunz multiplicity (see \Cref{HK multiplicity}). Indeed, for the pair $(R,I),$ the Hilbert-Kunz multiplicity is the value of the Frobenius-Poincar\'e function at zero. In fact, we provide an explicit formula for the coefficients of the power series expansion of the Frobenius-Poincar\'e entire function around zero, from which  it is apparent that each of the coefficients of this power series is an invariant  generalizing the Hilbert-Kunz multiplicity (see \Cref{power series expansion}). Even when $R$ is a polynomial ring, there are  examples of ideals with the same  Hilbert-Kunz multiplicity but different  Frobenius-Poincar\'e  functions- see \Cref{fine structure}. Like the Hilbert-Kunz multiplicity itself, the Frobenius-Poincar\'e function of $(R,I)$ depends only on the tight closure of $I$ in the ring $R$, as we prove in \Cref{dependence on tight closure}.\\\\
  

The information carried by the Frobenius-Poincar\'e function can also be understood in terms of homological data associated to the pair $(R,I)$. For example,
when $R$ is a polynomial ring (or more generally, when $R/I$ has finite projective dimension), we prove
 that 
the  Frobenius-Poincar\'e function has the form 
\begin{equation}\label{structure}
\frac{Q(e^{-iy})}{(iy)^{\text{dim} R}}
\end{equation}
 where $Q$ is a polynomial whose  coefficients are explicitly determined by the graded Betti numbers of $R/I$; see \Cref{finite projective dimesnion case}. More generally,
 for an arbitrary graded pair $(R,I)$, the Frobenius-Poincar\'e function of $(R,I)$ can be described in terms of the sequence of graded Betti numbers of $\frac{R}{\bp{I}{n}}$:\\\\
 \textbf{Theorem A:} Let $S \hookrightarrow R$ be a graded Noether normalization. Set $\mathbb{B}^S(j,n)= \sum \limits_{\alpha=0}^{\infty}(-1)^ \alpha \lambda(\text{Tor}_{\alpha}^S(\frac{R}{\bp{I}{n}},k)_j)$. Then the limiting function
 $$B^S(R,I)(y)= \underset{n \to \infty}{\lim}\underset{j \in \mathbb{N}}{\sum} \ \mathbb{B}^S(j,n)e^{-iyj/p^n}$$
 is entire. Furthermore, $\frac{B^S(R,I)(y)}{(iy)^{\text{dim}R}}$ is the Frobenius-Poincar\'e function of $(R,I)$ -- see \Cref{description in terms of Betti number} and \Cref{Betti numbers with respect to Noether normalization}.\\\\
 In a slightly different direction, we show that when $R$ is Cohen-Macaulay, the Frobenius-Poincar\'e entire function is a limit of a sequence of entire functions described in terms of Koszul homologies with respect to a homogeneous system of parameters for $R$, or alternatively,  Serre's intersection numbers, suitably interpreted; see \Cref{cohen-macaulay case}. \\
 
 The Frobenius-Poincar\'e function of $(R,I)$ turns out to be the Fourier transform of the Hilbert-Kunz density function of $(R,I)$ introduced by Trivedi in \cite{TriExist}, as we show in \Cref{FP function is fourier transform of density function}. Using Fourier transform, \textbf{Theorem A} allows us to describe the higher order weak derivatives of the Hilbert-Kunz density function in terms of the sequence graded Betti numbers of $\frac{R}{\bp{I}{n}}$ - see \Cref{differentiability of density function}. Such a description is not apparent in the existing theory of Hilbert-Kunz density functions. In fact, \textbf{Theorem A} relates the question on the order of smoothness of the Hilbert-Kunz density function raised in \cite{TriQuadric} to the question asking whether $B^S(R,I)(y)$ in \textbf{Theorem A} is bounded on the real line -- see \Cref{boundedness of holomorphic function defined by Betti numbers}. Although our work on Frobenius-Poincar\'e functions is inspired by Trivedi's remark that considering Fourier transforms of density functions might be useful (see \cite[page 3]{TriExist}), our proof of the existence and holomorphicity of the Frobenius-Poincar\'e function (\Cref{the main theorem}) is independent of \cite{TriExist}. When $R$ has dimension at least two and $R$ is strongly $F$-regular at each point on the punctured spectrum of $R$, the Hilbert-Kunz density function and hence the Frobenius-Poincar\'e function of $(R,I)$ captures the information of $F$-threshold of $I$- see Theorem 4.9 of \cite{TriFthreshold}. Recently Trivedi has used Hilbert-Kunz density functions to partially settle two conjectures on Hilbert-Kunz multiplicities of quadric hypersurfaces posed by Yoshida and Watanabe-Yoshida- see Theorem A and Theorem B in \cite{TriQuadric}. \\

We speculate that the entire functions that are Frobenius-Poincar\'e functions should have a special structure reflecting that each of these is determined by the data of a {\it{finitely generated}} module. Any such special structure will shed more light not only on the theory of Hilbert-Kunz multiplicities but also on the behaviour of graded Betti numbers of $\frac{R}{\bp{I}{n}}$ as $n$ changes. We ask whether Frobenius-Poincar\'e functions always have a form generalizing the expression (\ref{structure}) above; see  \Cref{the conjectural structure}. \Cref{the conjectural structure} is answered for one dimensional rings in \Cref{FP function in dimension one}. When $R$ is two dimensional, \Cref{the conjectural structure} is answered in \Cref{statement in the curve case}, where we show that the Frobenius-Poincar\'e function is described by the Harder-Narasimhan filtration on a sufficiently high Frobenius pullback of the syzygy bundle of $I$  on the curve $\text{Proj}(R)$ following \cite{TriCurves} and \cite{Brrationality}. The necessary background materials on vector bundles on curves and other topics are reviewed in \Cref{background material}. Also when the ideal $I$ is generated by a homogeneous system of parameters, our computation in \Cref{FP function wrt a homogeneous SOP} answers \Cref{the conjectural structure} positively.\\

We develop the theory of Frobenius-Poincar\'e functions more generally for triples $(M,R,I)$, where $M$ is a finitely generated $\Z$-graded $R$ module; see \Cref{main definition}. We show that the Frobenius-Poincar\'e function is additive on short exact sequences in \Cref{additivity over short exact sequences}. In addition to generalizing classical additivity formulas for Hilbert series and multiplicity, \Cref{additivity over short exact sequences} allows us to compute the Frobenius-Poincar\'e function of a graded ring with respect to an ideal generated by homogeneous system of parameters; see \Cref{FP function wrt a homogeneous SOP}.\\

\begin{notation} \label{Notation}
In this article, $k$ stands for a field.
By a \textit{finitely generated $\N$-graded $k$-algebra}, we mean an $\N$-graded commutative ring whose degree zero piece is $k$ and which is finitely generated over $k$. \\\\
For any ring $S$ containing $\mathbb{F} _p$, the Frobenius or $p$-th power endomorphism of $S$ is denoted by  $F_S$. The symbol $F_S^e$ will denote the $e$-times iteration of $F_S$. We set $S^{p^e}= F^e(S) \subseteq S$. For an ideal $J \subseteq S$, $JS^{p^e}$ is the image of $J$ in $S^{p^e}$ under the $p^e$-th power map. The ideal generated by $p^n$-th power of elements of $J$ in $S$ is denoted by $\bp{J}{n}$.\\\\
For an $S$-module $N$, we denote the Krull dimension of $N$ by $\text{dim}_S(N)$ or $\text{dim}(N)$ when the underlying ring $S$ is clear from the context. When $N$ has finite length, $\lambda_S(N)$ denotes the length of the $S$-module $N$. When $S=k$, simply $\lambda(N)$ will be used to denote the length.\\\\
Recall that an \textit{entire function} is a function holomorphic everywhere on the complex plane (see \cite{Ahlfors}, section 2.3).
\end{notation}
\noindent{ \large{\textbf{Acknowledgements:} }}  I am grateful to Karen Smith for her generosity with sharing ideas and suggestions which has improved the article substantially. I thank Mel Hochster and V. Trivedi for pointing out useful references; Jakub Witaszek, Sridhar Venkatesh for useful discussions; Ilya Smirnov and the anonymous referee for their comments. I also thank Daniel Smolkin, Janet Page, Jenny Kenkel, Swaraj Pande, Anna Brosowsky, Eamon Gallego for their questions, comments during a talk on an early version of this article. I was partially supported by NSF DMS grants \# 2101075, \# 1801697, NSF FRG grant \# 1952399 and Rackham one term dissertation fellowship while working on this article. 
\section{Background material}\label{background material}
In this section, we recall  some results, adapted to our setting, for future reference.  
\subsection{Hilbert-Kunz multiplicity}
Hilbert-Kunz multiplicity is a multiplicity theory in positive characteristic. We refer readers to \cite{HunekeExp} for a survey of this theory. In this subsection, $k$ is a field of characteristic $p>0$.
\begin{definition}\label{HK multiplicity}
Let $R$ be a finitely generated $\N$-graded $k$-algebra; $J$ be a homogeneous ideal such that $R/J$ has finite length. Given a finitely generated $\Z$-graded $R$-module $N$, the \textit{Hilbert-Kunz multiplicity} of the triple $(N,R,J)$ is defined to be the following limit
$$\underset{n \to \infty}{\lim}(\frac{1}{p^{n}})^{\text{dim}(N)}\lambda_R(\frac{N}{\bp{J}{n}N}).$$
Similarly one can define the Hilbert-Kunz multiplicity of a triple $(S,I,M)$- where $I$ is any finite co-length ideal in a Noetherian local ring $S$ and $M$ is a finitely generated $S$-module.
\end{definition}
\noindent The existence of the limit in the \Cref{HK multiplicity} was first established by Monsky (see \cite{MonExist}).\\\\
The Hilbert-Kunz multiplicity of any local ring is at least one. Moreover, under mild hypothesis, it is exactly one if and only if the ring is regular; see Theorem 1.5 of \cite{WY00} and \cite{HuYao}. These two facts suggest that Hilbert-Kunz multiplicity is a candidate for a multiplicity theory. In general, rings with Hilbert-Kunz multiplicity closer to one are interpreted to have better singularities; see \cite{BE} and \cite{GN}.\\\\

Unlike the usual Hilbert-Samuel function, the structure of the Hilbert-Kunz function $f(n)= \lambda(\frac{N}{\bp{J}{n}N})$ is rather elusive. We refer interested readers to \cite{SecCoef}, \cite{Teixeira}, \cite{Flag}.
\subsection{Betti numbers}\label{Homological}
We review results on graded Betti numbers which we use in \Cref{description using homological information}. References for most of these results are \cite{Serre}, and \cite{BH}. Recall that $R$ is a finitely generated $\N$-graded $k$-algebra (see \Cref{Notation}).\\\\
Given a finitely generated $\Z$-graded $R$-module $M$, one can choose a \textit{minimal graded free resolution} of $M$: this is a free resolution $(G_{\bullet}, d_{\bullet})$ of $M$ such that each $G_{n}$ is a graded free $R$-module, the boundary maps preserve graded structures, and the entries of the matrices representing boundary maps are forms of positive degrees.  As a consequence, $G_{r}\cong \underset {s \in \Z}{\oplus}R(-s)^{b_M^R(r,s)}$ where $b_M^R(r,s)= \lambda(\text{Tor}_r^R(k,M)_s)$.

\begin{definition}\label{graded Betti numbers notation}
Let $M$ be a finitely generated $\Z$-graded $R$-module. The $r$-th \textit{Betti number of $M$ with respect to $R$} is the rank of the free module $G_r$ at the $r$-th spot in a minimal graded free resolution of $M$, or equivalently, the length $\lambda(\text{Tor}_r^R(k,M))$.
\end{definition}
\begin{definition}\label{graded complete intersection}
The $\N$-graded ring $R$ is a \textit{graded complete intersection over $k$} if $R \cong \frac{k[X_1, \ldots, X_s]}{(f_1, \ldots, f_h)}$- where each $X_j$ is homogeneous of positive degree and $f_1, \ldots, f_h$ is a regular sequence consisting of homogeneous polynomials.
\end{definition}
\noindent We recall a special case of a result in \cite{Gul}.
\begin{lemma} \label{growth in the complete intersection case}
Let $R$ be a graded complete intersection over $k$. Then for any finitely generated $\Z$-graded $R$-module $M$, there is polynomial $P_M(t) \in \mathbb{Z}[t]$ such that for all $n$, $\lambda(\text{Tor}_n^R(M,k)) \leq P_M(n)$. 
\end{lemma}
\begin{proof}
Let $m$ be the homogeneous maximal ideal of $R$. Then $R_m$ is a local complete intersection as is meant in Corollary 4.2, \cite{Gul}. Since for all $n \in \N$, $\text{Tor}_n^R(M,k) \cong \text{Tor}^{R_m}_n(M_m,k)$, using Corollary 4.2, \cite{Gul}, we have a polynomial $\pi(t) \in \Z[t]$ and $r \in \N$ such that, $\sum \limits_{n=0}^{\infty}\lambda(\text{Tor}_n^R(M,k))t^n= \frac{\pi(t)}{(1-t^2)^r}$. The assertion in \Cref{growth in the complete intersection case} now follows by using the formal power series in $t$ representing $\frac{1}{(1-t^2)^r}$.
\end{proof}
\begin{lemma}\label{alternating sum of twists} Let $R$ be a finitely generated $\N$-graded $k$-algebra. Let $M$ be a finitely generated $\mathbb Z$-graded $R$-module. There is a positive integer $l$ such that given any integer $s$, $b^R_M(r,s):= \lambda(\text{Tor}^R_r(M,k)_s)=0$ for all $r \geq s+l$.
\end{lemma}
\begin{proof}
Pick a minimal free resolution $(G_ \bullet, d_ \bullet)$ of $M$, then $G_r \cong \underset{j \in \Z}{\oplus}R(-j)^{b^R_M(r,j)}$. Since the boundary maps of $G_ \bullet$ are represented by matrices whose entries are positive degree forms and have non-zero columns, $\phi(r):= \text{min}\{j \, | \, b^R_M(r,j) \neq 0\}$ is a strictly increasing function of $r$. So we can choose an integer $l$ such that $\phi(l)>0$. Again since $\phi(r)$ is strictly increasing, for all $r \geq j+l$, $\phi(r) \geq \phi(l)+j>j$. So given an integer $s$, $b^R_M(r,s)=0$ for all $r \geq s+l$.
\end{proof}
\begin{lemma}\label{radius of conv of twisted Poincare series}
 Let $R$ be a graded complete intersection over $k$ and $M$ be a finitely generated $\Z$-graded $R$-module. For a given integer $s$, let $\mathbb{B}_M(s)$ denote the sum $\underset{r \in \N}{\sum}(-1)^r{b_M^R(r,s)}$. Then the formal Laurent series $\underset{s \in \Z}{\sum} \mathbb{B}_M(s)t^s$ is absolutely convergent at every non-zero point on the open unit disk centered at the origin in $\mathbb{C}$.
\end{lemma}
\begin{proof}
By \Cref{alternating sum of twists}, there is an $l \in \N$ such that for any integer $s$, $|\mathbb{B}_M(s)| \leq \sum \limits_{j=0}^{s+l} \lambda(\text{Tor}_j^R(M,k))$. With $P_M$ the same as in \Cref{growth in the complete intersection case}, consider the polynomial 
\[Q_M(s)= \sum _{j=0}^{s+l}P_M(j).\]
Thus for $s \in \N$, $|\mathbb{B}_M(s)| \leq Q_M(s)$. So the radius of convergence of the power series $\sum \limits_{s=0}^{\infty}|\mathbb{B}_M(s)|t^s$ is at least one and the desired conclusion follows.
\end{proof}
\subsection{Hilbert series and Hilbert-Samuel multiplicities}\label{Hilbert series and Hilbert-Samuel multiplicity}
The references for this subsection are \cite{BH} and \cite{Serre}. Throughout, $R$ is a finitely generated $\N$-graded algebra over a field $k$.  Recall that the \textit{Hilbert series}(also called the \textit{Hilbert-Poincar\'e series}) of a finitely generated $\Z$-graded $R$-module $M$ is the formal Laurent series $H_M(t):= \underset{n \in \Z}{\sum}\lambda(M_n)t^n$. 

\begin{theorem}\label{rationality of Poincare series}(see Proposition 4.4.1, \cite{BH}) Let $M$ be a finitely generated $\Z$-graded $R$-module.
\begin{enumerate}
    \item There is a Laurent polynomial $Q_M(t) \in \mathbb{Q}[t,t^{-1}]$ such that,
    $$H_M(t)= \frac{Q_M(t)}{(1-t^{\delta_1}) \ldots (1-t^{\delta_{\text{dim(M)}}})}$$
    for some non-negative integers $\delta_1, \ldots, \delta_{\text{dim}(M)}$.
    \item The choice of $Q_M$ depends on the choices of $\delta_1, \ldots, \delta_{\text{dim}(M)}$. One can choose $\delta_1, \ldots, \delta_{\text{dim}(M)}$ to be the degrees of elements of $\frac{R}{\text{Ann(M)}}$ forming a homogeneous system of parameters. \footnote{A homogeneous system of parameters of a finitely generated $\N$-graded $k$-algebra $S$, is a collection of homogeneous elements $f_1, \ldots, f_{\text{dim}(S)}$ such that $\frac{S}{(f_1, \ldots, f_{\text{dim}(S)})}$ has a finite length (see \cite{BH}, page 35).}
\end{enumerate}
\end{theorem}
 In \Cref{Samuel multiplicity in terms of Hilbert-Poincare series}, we extend part of
Proposition 4.1.9 of \cite{BH}- where $R$ is assumed to be standard graded- to our setting. We use \Cref{Samuel multiplicity in terms of Hilbert-Poincare series} to define Hilbert-Samuel multiplicity of a finitely generated $\Z$-graded module over a graded ring- where the ring is not necessarily standard graded-in part (1), \Cref{HS multiplicity}. 
\begin{proposition}\label{Samuel multiplicity in terms of Hilbert-Poincare series}
Let $M$ be a finitely generated $\Z$-graded $R$-module of Krull dimension d. Denote the Poincar\'e series of $M$ by $H_M(t)$. 
\begin{enumerate}
    \item The limit $d! \underset{n \to \infty}{\lim}\frac{1}{n^d}(\underset{j \leq n}{\sum} \lambda(M_j))$ exists . The limit is denoted by $e_M$.
    \item The limit $\underset{t \to 1}{\lim}(1-t)^dH_M(t)$ is the same as $e_M$.
\end{enumerate}
\end{proposition}
\vspace{.2cm}
\begin{proof}[Proof of \Cref{Samuel multiplicity in terms of Hilbert-Poincare series}]
When $M$ has Krull dimension zero, the desired conclusion is immediate. So we assume that $M$ has a positive Krull dimension. We first prove (1).

Let $f_1, \ldots, f_d$ be a homogeneous system of parameters of $\frac{R}{\text{Ann}(M)}$ of degree $\delta_1, \ldots, \delta_d$ respectively. Set $\delta$ to be the product $\delta_1 \ldots \delta_d$ and $g_j= f_j^{\frac{\delta}{\delta_j}}$. Then each of $g_1, \ldots, g_d$ has degree $\delta$ and these form a homogeneous system of parameters of $\frac{R}{\text{Ann}(M)}$. Denote the $k$-subalgebra generated by $g_1, \ldots, g_d$ by $S$. We endow $S$ with a new $\N$-grading: given a natural number $n$, declare the $n$-th graded piece of $S$ to be
$$S_n := S \cap (\frac{R}{\text{Ann}(M)})_{\delta n}.$$
From now on, by the grading on $S$ we refer to the grading defined above. Note that $S$ is a standard graded $k$-algebra. Now for each $r$, where $0 \leq r < \delta$, set
$$M^{r}= \underset{n \in \Z}{\oplus} M_{n \delta+ r}.$$
Given an $r$ as above, we give $M^{r}$ a $\Z$-graded structure by declaring the $n$-th graded piece of $M^r$ to be $M_{n \delta+r}$. Then each $M^{r}$ is a finitely generated $\Z$-graded module over $S$. Since $S$ is standard graded, for each $r$, $0 \leq r < \delta$, the limit
$$\underset{n \to \infty}{\lim}\frac{1}{n^{d-1}}\lambda(M^r_n)$$
exists (see Theorem 4.1.3, \cite{BH}). This implies the existence of a constant $C$ such that $\lambda(M_n) \leq Cn^{d-1}$ for all $n$. So the sequence $\frac{d!}{n^d}(\underset{j \leq n}{\sum} \lambda(M_j) )$ converges if and only if the subsequence $$\left(\frac{d!}{(\delta n+ \delta -1)^d}(\underset{j \leq (n+1) \delta -1}{\sum}\lambda(M_j)) \right)_n$$ converges. Now we show that the above subsequence is convergent by computing its limit. 
\begin{equation}\label{existence of limit}
\begin{split}
& d!\underset{n \to \infty}{\lim}\frac{1}{(\delta n+ \delta -1)^d}(\underset{j \leq (n+1) \delta -1}{\sum}\lambda(M_j))\\
    & = \underset{n \to \infty}{\lim} \frac{n^d}{(\delta n + \delta -1)^d} [\sum \limits_{r=0}^{\delta -1}\frac{d!}{n^d}(\underset{j \leq n }{\sum}\lambda(M^r_j))]
\end{split}
\end{equation}
Since each $M^r$ where $0 \leq r \leq \delta -1$ is a finitely generated module over the standard graded ring $S$, by Proposition 4.1.9 and Remark 4.1.6 of \cite{BH}, the last limit in \eqref{existence of limit} exists and
\begin{equation}\label{limit of subsequence equal to HS multiplicity}
    e_M = \frac{e_{M^0}+ \ldots+ e_{M^{\delta -1}}}{\delta ^d}.
\end{equation}
For (2), note that
\begin{equation}\label{limit and Poincare series}
\begin{split}
\underset{t \to 1}{\lim}(1-t)^dH_M(t) &= \underset{t \to 1}{\lim} \sum \limits_{r=0}^{\delta -1}(1-t)^dH_{M^r}(t^ \delta)t^r\\
&= \underset{t \to 1}{\lim}\frac {\sum \limits_{r=0}^{\delta -1}(1-t^\delta)^dH_{M^r}(t^ \delta)t^r}{(1+t+ \ldots+ t^{\delta -1})^d}.
\end{split}
\end{equation}
Again, since each $M^r$ is a finitely generated module over the standard graded ring $S$, by Proposition 4.1.9 and Remark 4.1.6 of \cite{BH}
$$e_{M^r}= \underset{t \to 1}{\lim}(1-t)^d H_{M^r}(t).$$
So from \eqref{limit and Poincare series} and \eqref{limit of subsequence equal to HS multiplicity}, we get
\begin{equation*}
 \underset{t \to 1}{\lim}(1-t)^d H_M(t)= \frac{e_{M^0}+ \ldots+ e_{M^{\delta -1}}}{\delta ^d}= e_M.
\end{equation*}
\end{proof}
\begin{definition}\label{HS multiplicity}
Let $M$ be a finitely generated $\Z$-graded $R$-module of Krull dimension $d$.
\begin{enumerate}
\item The \textit{Hilbert-Samuel multiplicity} of $M$ is defined to be the limit
$$d!\underset{n \to \infty}{\lim}\frac{1}{n^d}(\underset{j \leq n}{\sum}\lambda(M_j))$$ and denoted by $e_M$. The limit exists by (1) \Cref{Samuel multiplicity in terms of Hilbert-Poincare series}.
\item Given a homogeneous ideal $I$ of finite co-length, the \textit{Hilbert-Samuel multiplicity of $M$ with respect to $I$} is defined to be the limit:
$$d!\underset{n \to \infty}{\lim}\frac{1}{n^d}\lambda(\frac{M}{I^nM}).$$
\end{enumerate}
\end{definition}

\noindent
\begin{proposition}\label{HS multiplicity with respect to SOP}
Let $f_1, \ldots, f_d$ be a homogeneous system of parameters of $R$ of degree $\delta_1, \ldots, \delta_d$ respectively. Then the Hilbert-Samuel multiplicity of $R$ with respect to $(f_1, \ldots, f_d)$ (see \Cref{HS multiplicity}) is $\delta_1 \ldots \delta_d e_R$.
\end{proposition}
\begin{proof}
By Proposition 2.10 of \cite{HunekeTakagi}, the desired multiplicity is $\delta_1 \ldots \delta_d\underset{t \to 1}{\lim}(1-t)^dH_R(t)$, which by \Cref{Samuel multiplicity in terms of Hilbert-Poincare series} is $\delta_1 \ldots \delta_d e_R$.
\end{proof}
\subsection{Vector bundles on curves}\label{vector bundles on curve}
In this subsection, $C$ stands for a curve, where by a curve we mean a one dimensional, irreducible smooth projective variety over an algebraically closed field; the genus of $C$ is denoted by $g$. A \textit{vector bundle on $C$} means a locally free sheaf $\mathcal{O}_C$-modules of finite constant rank. Morphisms of vector bundles are a priori morphisms of $\mathcal O_C$-modules. We recall some results on vector bundles on $C$  which we use in \Cref{in dimension two}. For any unexplained terminology, readers are requested to turn to \cite{Hart} or \cite{Potier}.
\begin{definition}\label{subbundle} Let $\mathcal F$ be a coherent sheaf on the curve $C.$
\begin{enumerate}
    \item The rank of $\mathcal F,$ denoted by $\text{rk}(\mathcal{F})$, is the dimension of the stalk of $\mathcal F$ at the generic point of $C$ as a vector space over the function field of $C$.
    \item The degree of $\mathcal F,$ denoted by $\text{deg}(\mathcal{F})$, is 
    defined as  $h^0(C, \mathcal{F})- h^1(C, \mathcal{F})- \text{rk}(\mathcal{F})(1-g)$.
    \item The \textit{slope} of $\mathcal{F}$,  denoted by $\mu(\mathcal{F})$, is the ratio
    $\frac{\text{deg}(\mathcal{F})}{\text{rk}(\mathcal{F})}$.
     By convention, $\mu(\mathcal F)=\infty$ if 
    $\text{rk}(\mathcal{F})=0$.
   \end{enumerate}
\end{definition}
\begin{definition}
A vector bundle $\E$ on $C$ is called \textit{semistable} if for any nonzero coherent subsheaf $\mathcal{F}$ of $\E$, $\mu(\mathcal{F}) \leq \mu(\E)$.  
\end{definition}
\begin{theorem}\label{HN filtration}(see \cite[Prop 1.3.9]{HN})
Let $\E$ be a vector bundle on $C$. Then there exists a unique filtration:
$$0=\E_0 \subset \E_1 \subset \ldots \subset \E_t \subset \E_{t+1}= \E$$
such that,
\begin{enumerate}
    \item All the quotients $\E_{j+1}/\E_j$ are non-zero, semistable vector bundles.
    \item For all $j$, $\mu(\E_j/ \E_{j-1})> \mu(\E_{j+1}/\E_j)$.
\end{enumerate}
This filtration is called the \textit{Harder-Narasimhan filtration} of $\E$.
\end{theorem}
\medskip
\begin{proposition}\label{minimum slope negative implies no global sections}
Let $0=\E_0 \subset \E_1 \subset \ldots \subset \E_t \subset \E_{t+1}= \E$ be the HN filtration on $\E$. If the slope of $\E_1$ is negative, $\E$ cannot have a non-zero global section.
\end{proposition}
\begin{proof}
On the contrary, assume that $\E$ has a non-zero global section $s$. Let $\lambda_s: \mathcal{O}_C \rightarrow \E$ be the non-zero map induced by $s$. Let $b$ be the largest integer such that the composition $ \mathcal{O}_C \xrightarrow{\lambda_s} \E \rightarrow \E/\E_b$ is non-zero. Then $\lambda_s$ induces a non-zero map from $\mathcal{O}_C$ to $\E _{b+1}/\E_b$, whose image $\mathcal{L}$ is a line bundle with a non-zero global section. So the slope of $\mathcal{L}$ is positive. On the other hand, since $\mathcal{L}$ is a non-zero subsheaf of the semistable sheaf $\E _{b+1}/\E_b$, $\mu({\mathcal{L}})< \mu({\E _{b+1}/\E_b})$. Since $\mu (\E _{b+1}/\E_b) < \mu(\E _1)$, the slope of $\E _{b+1}/\E_b$ is negative ; so $\mathcal{L}$ cannot have a positive slope.
\end{proof}
\begin{lemma}\label{properties of HN filtration}
\begin{enumerate}
    \item For a coherent sheaf of $\mathcal{O}_C$ modules $\mathcal{F}$ and a line bundle $\mathcal{L}$, $\mu(\mathcal{F}\otimes \mathcal{L}) = \mu(\mathcal{F})+ \text{deg}(\mathcal{L})$. Here we stick to the convention that the sum of $\infty$ and a real number is $\infty$.
    \item Tensor product of a semistable vector bundle and a line bundle is semistable.
    \item Given a vector bundle $\E$ and a line bundle $\mathcal{L}$ on $C$, the HN filtration on $\E \otimes \mathcal{L}$ is obtained by tensoring the HN filtration on  $\E$ with $\mathcal{L}$.
\end{enumerate}
\end{lemma}
\begin{proof}
Because (3) follows from (1) and (2); and assertion (2) follows from (1), it suffices to prove (1). For (1), it is enough to show that 
\begin{equation}\label{ses}
\text{deg}(\mathcal{F}\otimes \mathcal{L})= \text{deg}(\mathcal{F})+ \text{rk}(\mathcal{F})\text{deg}(\mathcal{L}) \ .
\end{equation}
 This is clear when $\text{rk}(\mathcal{F})=0$. In the general case, take the short exact sequence of sheaves $0 \rightarrow \mathcal{F}' \rightarrow \mathcal{F} \rightarrow \mathcal{F}'' \rightarrow 0$, where $\mathcal{F}'$ is the torsion subsheaf of $\mathcal F$ and
$\mathcal{F}'':= \mathcal{F}/\mathcal{F}'$ is a vector bundle; note that the rank of $\mathcal{F}'$ is zero.
 Since degree is additive over short exact sequences (see section 2.6, \cite{Potier}),
 it suffices to show (\ref{ses}) when $\mathcal F=\mathcal F''$, that is, when $\mathcal F$ is locally free. In this case,
 $\text{deg}(\mathcal{F}\otimes \mathcal{L})= \text{deg}(\text{det}(\mathcal{F}\otimes \mathcal{L}))= \text{deg}(\text{det}(\mathcal{F})\otimes \mathcal{L}^{\otimes{\text{rk}}(\mathcal{F})})$ (for e.g. by Theorem 2.6.9 of \cite{Potier}), so  $\text{deg}(\mathcal{F}\otimes \mathcal{L})= \text{deg}(\mathcal{F})+ \text{rk}(\mathcal{F})\text{deg}(\mathcal{L})$ by  Theorem 2.6.3, \cite{Potier}.
\end{proof}
\vspace{.2cm}
In the next lemma, for a sheaf of $\mathcal{O}_C$-modules $\mathcal{F}$, $\mathcal{F}^ \vee$ denotes the \textit{dual} sheaf $\underline{\text{Hom}}_{\mathcal{O}_C}(\mathcal{F}, \mathcal{O}_C)$.
\begin{lemma}\label{filtration on the dual}
\begin{enumerate}
    \item The dual of a semistable vector bundle is semistable.
    \item Let $0=\E_0 \subset \E_1 \ldots \subset \E_t \subset \E_{t+1}= \E$ be the HN filtration on a vector bundle $\E$. For $j$ between $0$ and $t+1$, set $K_j= \text{ker}(\E^ \vee \rightarrow \E_{t+1-j}^\vee)$. Then 
$$0=K_0 \subset K_1 \subset \ldots \subset K_t \subset K_{t+1}= \E^ \vee$$
is the HN filtration on $\E^ \vee$.
\end{enumerate}
\end{lemma}
\begin{proof}
(1) Let $\mathcal{F}$ be semistable vector bundle and $\mathcal{G}$ be a non-zero subsheaf of $\mathcal{F}^ \vee$. We show that $\mu(\mathcal{G})\leq \mu(\mathcal{F}^ \vee)$- this is clear when $\text{rk}(\mathcal{F}^ \vee/\mathcal{G})=0$ as $\text{deg}(\mathcal{F}^ \vee)= \text{deg}(\mathcal{G})+ \text{deg}(\mathcal{F}^ \vee/\mathcal{G})$ and $\text{rk}(\mathcal{F}^ \vee)= \text{rk}(\mathcal{G})$. If $\text{rk}(\mathcal{F}^ \vee/\mathcal{G})$ is not zero, set $\mathcal{G}'$ to be the inverse image of $(\mathcal{F}^ \vee/\mathcal{G})_{\text{tor}}$: the torsion subsheaf of $\mathcal{F}^ \vee/\mathcal{G}$, under the quotient map $\mathcal{F}^ \vee \rightarrow \mathcal{F}^ \vee/\mathcal{G}$. Then $ \mathcal{G}'/\mathcal{G} \cong (\mathcal{F}^ \vee/\mathcal{G})_{\text{tor}}$. So $\mathcal{G}$ and $\mathcal{G}'$ have the same rank. Since $\text{deg}(\mathcal{G}') \geq \text{deg}(\mathcal{G})$, it is enough to show that $\mu(\mathcal{G}') \leq \mu(\mathcal{F})$. Since $\mathcal{F}/\mathcal{G}'$ is a vector bundle, after dualizing we get an exact sequence:
$$0 \rightarrow (\mathcal{F}/\mathcal{G}')^ \vee \rightarrow \mathcal{F} \rightarrow (\mathcal{G}')^ \vee \rightarrow 0.$$
Since $\mathcal{F}$ is semistable, $\mu((\mathcal{G}')^ \vee) \geq \mu(\mathcal{F})$- see section 5.3, \cite{Potier}. Since for a vector bundle $\mathcal{E}$, $\text{deg}(\mathcal{E}^ \vee)= -\text{deg}(\mathcal{E})$, we have $\mu(\mathcal{G}') \leq \mu(\mathcal{F}^ \vee)$.\\\\
(2) $K_{j+1}/ K_j \cong (\E_{t+1-j}/\E_{t-j})^ \vee$, so by (1) $(\E_{t+1-j}/\E_{t-j})$ is semistable. Moreover, since $\mu(K_{j+1}/ K_j)= - \mu(\E_{t+1-j}/\E_{t-j})$, slopes of $K_{j+1}/ K_j$ form a decreasing sequence.
\end{proof}
\vspace{.2cm}
Let $C$ be a curve over an algebraically closed field of positive characteristic. Let $f$ be the \textit{absolute Frobenius} endomorphism of $C$. Since $C$ is smooth, $f$ is flat map (see Theorem 2.1, \cite{Kunz}). So the pullback of the HN filtration on a given vector bundle gives a filtration of the pull back bundle by subbundles- in general this is not the HN filtration on the pullback bundle. 
\begin{theorem}\label{Langer}(see Theorem 2.7 \cite{Langer})
Let $\E$ be a vector bundle on a curve $C$. Then there is an $n_0 \in \N$ such that for $n \geq n_0$, the HN filtration on $(f^n)^*(\E)$ is the pullback of the HN filtration on $(f^{n_0})^*(\E)$ via $f^{n-n_0}$.
\end{theorem}

\section{Existence of Frobenius-Poincar\'e functions}   
In this section, we define the Frobenius-Poincar\'e function associated to a given triple $(M,R,I)$, where $R$ is a finitely generated $\N$-graded $k$-algebra -- $k$ has characteristic $p>0$,  (see \Cref{Notation}), $I$ is a homogeneous ideal of finite co-length and $M$ is a finitely generated $\Z$-graded $R$-module. In \Cref{the main theorem} we prove that Frobenius-Poincar\'e functions are \textit{entire functions}.\\

Given $(M,R,I)$ as above and a non-negative integer $d$, define a sequence of  functions $F_n(M, R, I, d)$, where for a complex number $y$,
\begin{equation}\label{the sequence}
F_n(M,R,I,d)(y)= \expand{\frac{M}{\bp{I}{n}M}}{d}.
\end{equation}
Since $M/I^{[p^n]}M$ has only finitely many non-zero graded pieces, each $F_n(M,R,I,d)$ is a polynomial in $e^{-iy/p^n}$, hence is an entire function. When the context is clear, we suppress one or more of the parameters among $M, R, I, d$ in the notation $F_n(M,R,I,d)$. Whenever there is no explicit reference to the parameter $d$ in $F_n(M,R,I,d)$, it should be understood that $d= \text{dim}(M)$.\\\\
The goal in this section is to prove the following result:
\begin{theorem}\label{the main theorem}
Fix a triple $(M,R,I)$, where $R$ is a finitely generated $\N$-graded $k$-algebra (see \Cref{Notation}), $I$ is a finite co-length homogeneous ideal, and $M$ is a finitely generated $\Z$-graded $R$-module.  The sequence of functions $(F_n(M,R,I, \text{dim}(M)))_n$, where
$$F_n(M,R,I, \text{dim}(M))(y)= \expand{\frac{M}{\bp{I}{n}M}}{\text{dim}(M)},$$
converges for every complex number $y$. Furthermore, the convergence is uniform on every compact subset of the complex plane and the limit is an entire function.
\end{theorem}
\Cref{the main theorem} motivates the next definition.
\begin{definition}\label{main definition} The \textit{Frobenius-Poincar\'e function} of the triple $(M,R,I)$ is the limit of the  convergent sequence of functions $$(F_n(M,R,I, \text{dim}(M)))_n$$ as defined in \Cref{the main theorem}. The Frobenius-Poincar\'e function of the triple $(M,R,I)$ is denoted by $F(M,R,I, \text{dim}(M))$ or alternately $F(M,R,I)$ or just $F(M)$ when the other parameters are clear from the context.
\end{definition}
\vspace{.5cm}

\noindent Before giving examples of Frobenius-Poincar\'e functions, we single out a limit computation.
\begin{lemma}\label{the essential limit}
Given a complex number $a$, the sequence of functions $(p^n(1-e^{-iay/p^n}))_n$ converges to the function $g(y)= iay$ and the convergence is uniform on every compact subset of the complex plane.
\end{lemma}
\begin{proof}

Using the power series representing $e^z$ around the origin, we get
\[p^n(1-e^{-iay/p^n})= iay+ \sum_{j=2}^{\infty}\frac{(-iay)^j}{j!(p^n)^{j-1}}.\]
So
$$|p^n(1-e^{-iay/p^n})-iay| \leq \frac{1}{p^n}\sum_{j=2}^{\infty}\frac{|ay|^j}{j!} \leq \frac{1}{p^n}e^{|ay|}.$$
Since $e^{|ay|}$ is bounded on any compact subset, we get the desired unifrom convergence.
\end{proof}
\begin{example}\label{one variable polynomial case}
Consider the $\N$-graded polynomial ring in one variable $R= \mathbb{F}_p[X]$ where $X$ has degree $\delta \in \N$ and elements of $\mathbb{F}_p$ have degree zero. Take $I$ to be the ideal generated by $X$. Then, for any nonzero $y \in \mathbb{C}$,
$$F_n(R)(y)= \frac{1}{p^n} \sum\limits_{j=0}^{p^n-1}e^{-iy\delta j/p^n}= \frac{1}{p^n}\frac{1-e^{-i\delta y}}{1-e^{-i\delta y/p^n}}.$$ 
Taking limit as $n$ goes to infinity in the above equation and using \Cref{the essential limit}, we get that for a non-zero complex number $y$, $F(R)(y)= \frac{1-e^{-i\delta y}}{i\delta y}$. Note that $\frac{1-e^{-i\delta y}}{i\delta y}$ can be extended to an analytic function with value one at the origin. Since $F_n(0)=1$ for all $n$, $F(R)$ and the analytic extension of $\frac{1-e^{-i\delta y}}{i\delta y}$ are the same function.\\\\
Similar computation shows that the Frobenius-Poincar\'e function of the triple $(R, I, R)$ where $I$ is the ideal generated by $X^t$ is $\frac{1-e^{-i\delta t y}}{i \delta y}$.
\end{example}
\begin{example}\label{multivariable polynomial ring}
Take $R= \mathbb{F}_p[X_1, \ldots, X_n]$ with the grading assigning degree $\delta_j$ to $X_j$ and degree zero to the elements of $\mathbb{F}_p$. Since as graded rings,
$$\frac{\mathbb{F}_p[X_1, \ldots, X_d]}{(X_1^{p^n}, X_2^{p^n}, \ldots, X_d^{p^n})} \cong \mathbb{F}_p[X_1]/(X_1^{p^n}) \otimes_ {\mathbb{F}_p}\mathbb{F}_p[X_2]/(X_2^{p^n}) \otimes_{\mathbb{F}_p} \ldots \otimes_ {\mathbb{F}_p}\mathbb{F}_p[X_d]/(X_d^{p^n}),$$
we have $F_n(R, R, (X_1, \ldots, X_n))(y)= F_n(\mathbb{F}_p[X_1], \mathbb{F}_p[X_1], (X_1)). \ldots F_n(\mathbb{F}_p[X_d], \mathbb{F}_p[X_d], (X_d))$.
So from \Cref{one variable polynomial case}, it follows that $F(R, R, (X_1, \ldots, X_d))(y)= \prod \limits_{j=1}^{d} (\frac{1-e^{-i \delta_j y}}{i\delta_j y})$.
\end{example}
\begin{remark}\label{motivation}
 The fractional exponents of exponentials in the definition of Frobenius-Poincar\'e function of a triple $(M,R,I)$ comes from the natural $\frac{1}{p^n}\Z$ grading on $F^n_*(M)$. Here $F^n_*(M)$ is the abelian group $M$ enowed with the $R$-module structure coming from the restriction of scalars via the $n$-th iteration of Frobenius $F^n_R: R \rightarrow R$. The $\frac{1}{p^n} \Z$-graded structure on $F^n_*(M)$ is as follows: for an integer $m$, $F^n_*(M)_{m/p^n}=M_m$. For example, when $R$ is a domain, the $\frac{1}{p^n}\Z$-grading on $F^n_*(R)$ described above is the one obtained by importing the natural $\frac{1}{p^n}\Z$-grading on $R^{1/p^n}$ via the $R$-module isomorphism $R^{1/p^n} \rightarrow F^n_*(R)$ given by the $p^n$-th power map.\\\\
 Note that as $\frac{1}{p^n}\Z$-graded modules $F^n_*(\frac{M}{\bp{I}{n}M})\cong F^n_*(M)\otimes_R R/I$. Set $\ell$ to be the degree of the field extension $k^p \subseteq k$ and $\Z[1/p]$ to be the ring $\Z$ with $p$ inverted. Alternatively $F_n$ can be expressed as, 
 $$F_n(M,R,I)(y)= (\frac{1}{\ell p^{\text{dim}(M)}})^n\underset{t \in \Z[1/p]}{\sum}\lambda_k \left( (F^n_*(M) \otimes R/I)_t \right) e^{-ity}.$$
 That is, $F_n$ is the \textit{Hilbert series} (see \Cref{Hilbert series and Hilbert-Samuel multiplicity}) of $F^n_*(M) \otimes R/I$  normalized by $(\frac{1}{\ell p^{\text{dim(M)}}})^n$ in the `variable' $e^{-iy}$. The associated Frobenius-Poincar\'e function is the limit of these normalized Hilbert series.
\end{remark}
\begin{remark}\label{reduction to the F-finite case}
Given a field extension $k \subseteq k'$, $R \otimes_k k'$ is a finitely generated $\N$-graded $k'$-algebra. Note that $F_n(M,R,I,d)= F_n(M \otimes_k k', R \otimes_k k', I \otimes_k k', d)$. Thus for any complex number $y$, $(F_n(M, R,I,d)(y))_n$ converges if and only if $(F_n(M\otimes_k k', R\otimes_k k', I\otimes_k k',d)(y))_n$ converges. So in the proof of \Cref{the main theorem} without loss of generality we can assume that $k^p \subseteq k$ is a finite extension or even algebraically closed. thus we can assume $k=k^p$. 
\end{remark}
The remainder of the section is dedicated to the proof of \Cref{the main theorem}. The proof has two main steps. First, we reduce the problem to the case where $R$ is an $\N$-graded domain and $M=R$ as a graded module- this reduction step is achieved in \Cref{reduction to domain}. Then we show that when $R$ is a domain, $F_n(R,R,I)(y)$ is uniformly Cauchy on every compact subset of $\mathbb{C}$. Thus $F_n(R,R,I)$ converges uniformly on every compact subset of the complex plane. The analyticity of the limiting function then follows from Theorem 1 in Chapter 5 of \cite{Ahlfors}: \textit{a sequence of holomorphic functions on an open subset $U \subseteq \mathbb{C}$, which converges uniformly on every compact subset of $U$ has a holomorphic limiting function}.\\\\
One of the purposes of the next result is to show that in the definition of $F_n$ in \eqref{the sequence}, it is enough to take the sum over indices $j$, where $|j|/p^n$ is bounded by a constant which is independent of $n$.
\begin{lemma}\label{support}
Let $M$ be a finitely generated $\Z$-graded $R$-module. Given $i \in \N$, there exists a positive integer $C$ such that for all $n$ whenever $|j| \geq Cp^n$, $(\text{Tor}^R_i(M, \frac{R}{\bp{I}{n}R}))_j$ is zero.
\end{lemma}
\begin{proof}
We claim that it is sufficient to prove existence of some $C$ such that $(\frac{R}{\bp{I}{n}R})_{j}=0$ for $|j| \geq Cp^n$. To see this, consider a \textit{graded minimal free resolution} (see section 2.2) of the $R$-module $M$:
$$\begin{tikzcd}
{} \arrow[r, "\partial_{i+1}"] & F_i \arrow[r, "\partial_i"] & \ldots \arrow[r, "\partial_2"] & F_1 \arrow[r, "\partial_1"] & F_0 \arrow[r, "\partial_0"] & M \arrow[r] & 0.
\end{tikzcd}$$
As a graded module,
$\text{Tor}^R_i(M, \frac{R}{\bp{I}{n}R})$ is a subquotient of $ \text{Tor}^R_0(F_i, \frac{R}{\bp{I}{n}}) \cong F_i/\bp{I}{n}F_i$ and $F_i$ is a direct sum of finitely many graded modules of the form $R(l)$, so our claim follows.

Now choose homogeneous elements of positive degrees $r_1, r_2, \ldots, r_s$ of $R$ such that $R= k[r_1, \ldots, r_s]$. Let $\Delta= \text{max}\{\text{deg}(r_i)\}$ and $\delta = \text{min}\{\text{deg}(r_i)\}$. Denote the ideal of generated by homogeneous elements of degree at least $t$ by $R_{\geq t}$. Then note that for every $n \in \N$, 
\begin{equation}\label{eq9}
R_{\geq n\Delta} \subseteq (R_{\geq \delta})^n.
\end{equation}
Indeed,
for each $r \in R_t$ where $t \geq n\Delta$, we can choose homogeneous elements $\lambda_1, \ldots, \lambda_s$ of $R$ such that
$$r= \lambda_1 r_1 + \ldots+ \lambda_sr_s.$$
Then each $\lambda_i \in R_{\geq t- \Delta} \subseteq R_{(n-1)\Delta}$.  Since each $r_i$ is in $R_{\geq \delta}$, the claimed assertion in (\ref{eq9}) follows by induction on $n$.

Pick $m_0$ such that $(R_{\geq \delta})^{m_0} \subseteq I$. Suppose the minimal number of homogeneous generators of $I$ is $\mu$. Using \eqref{eq9}, we get
$$R_{\geq m_0\mu p^n \Delta} \subseteq (R_{\geq \delta})^{m_0 \mu p^n} \subseteq I^{\mu p^n} \subseteq \bp{I}{n}.$$
Therefore, if we set $C= m_0 \mu \Delta$, for $m \geq Cp^n$, $(\frac{R}{\bp{I}{n}R})_m =0$.
\end{proof}

\noindent We now bound the asymptotic growths of two length functions.
\begin{lemma}\label{bounded growth}
Let $(S,m)$ be Noetherian local ring containing a field of positive characteristic $p$, $J$ be an $m$-primary ideal. For any finitely generated $S$-module $N$, there exist positive constants $C_1, C_2$ such that for all $n \in \N$,
$$\lambda_S(N/\bp{J}{n}N) \leq C_1 (p^n)^{\text{dim}(N)} \, \, \, \, \text{and} \, \,
\hspace{.5cm} \lambda_S(\text{Tor}^S_1(N, \frac{S}{\bp{J}{n}S})) \leq C_2 (p^n)^{\text{dim}(N)}.$$
\end{lemma}
\vspace{.5cm}
\begin{proof}
The assertion on the growth of $\lambda_S(N/\bp{J}{n}N)$ is standard, for example see Lemma 3.5 of \cite{HunekeExp} for a proof.  For the other assertion, we present a simplified version of the argument in Lemma 7.2 of \cite{HunekeExp}. Suppose that $J$ is generated by $f_1, \ldots, f_{\mu}$.
 Let $K_{\bullet}$ be the \textit{Koszul complex} of $f_1, f_2, \ldots, f_{\mu}$. Recall that the \textit{Frobenius functor} $\mathfrak{F}$, from the category of $S$-modules to itself, is the scalar extension via the Frobenius $F_S: S \rightarrow S$ (see page 7, \cite{Phantom}). Let $\mathfrak{F}^n(K_{\bullet})$ stand for the complex of $S$-modules obtained by applying $n$-th iteration of $\mathfrak{F}$ to the terms and the boundary maps of $K_{\bullet}$. Then the part $\mathfrak{F}^n(K_1) \rightarrow \mathfrak{F}^n(K_0)$ of $\mathfrak{F}^n(K_{\bullet})$ can be extended to a free resolution of the $S$-module $S/\bp{J}{n}S$. So  $\text{Tor}^S_1(N, \frac{S}{\bp{J}{n}S})$ is isomorphic to a quotient of $H_1(N \otimes_S \mathfrak{F}^n(K_{\bullet}))$. Hence $\lambda(\text{Tor}^S_1(N, \frac{S}{\bp{J}{n}S})) \leq \lambda_S(H_1(N \otimes_S \mathfrak{F}^n(K_{\bullet})))$. The conclusion follows from Theorem 6.6 of \cite{Phantom}, which guarantees that there is a constant $C_2$ such that for all $n$, $\lambda_S(H_1(N \otimes_S \mathfrak{F}^n(K_{\bullet}))) \leq C_2(p^n)^{\text{dim}(N)}$.\end{proof}
\noindent The next result bounds the growth of the sequence $(F_n(M,R,I,d))_n$ on a given compact subset of the complex plane.
\begin{proposition}\label{FP functions are bounded on a compact subset}Let $M$ be a finitely generated $\Z$-graded $R$-module. 
\begin{enumerate}
\item Given a compact subset $A \subseteq \mathbb{C}$, there exists a constant $D$ such that  for all $y\in A$ and $n \in \N$,
$$|F_n(M,R,I,d)(y)| \leq (\frac{1}{p^n})^{d- \text{dim}(M)}D.$$
\item On a given compact subset of $\mathbb{C}$, when $d \geq \text{dim}(M)$, the sequence $(F_n(M,R,I,d))_n$ is uniformly bounded and when $d> \text{dim}(M)$, the sequence $(F_n(M,R,I,d))_n$ uniformly converges to the constant function taking value zero.
\end{enumerate}
\end{proposition}
\begin{proof}
Assertion (2) is immediate from (1).

We now prove (1). Let $C'$ be a positive constant such that for $y \in A$, $|y| \leq C'$. According to \Cref{support}, we can choose a positive constant $C$ such that for all $n \in \N$ and $|j| > Cp^n$, $(M/\bp{I}{n}M)_j=0$. For $y \in A$ and $|j| \leq Cp^n$,
\begin{equation}\label{bound on the sequence}
|e^{-iyj/p^n}|^2= e^{-iyj/p^n}. e^{i\bar{y}j/p^n}= e^{2\text{Re}(-iyj/p^n)}\leq e^{2|-iyj/p^n|} = e^{2|y||j|/p^n} \leq e^{2CC'} \ .
\end{equation}
Now note that,
$$|F_n(M,d)(y)|= |(\frac{1}{p^n})^d \underset{|j|\leq Cp^n}{\sum}\lambda \left((\frac{M}{\bp{I}{n}M})_j \right)e^{-iyj/p^n}| \leq (\frac{1}{p^n})^{d}\underset{|j|\leq Cp^n}{\sum}\lambda \left((\frac{M}{\bp{I}{n}M})_j\right)|e^{-iyj/p^n}|,$$
So using \eqref{bound on the sequence} first and then \Cref{bounded growth}, we get,
$$|F_n(M,d)(y)| \leq (\frac{1}{p^n})^{d} \lambda(\frac{M}{\bp{I}{n}M})e^{CC'} \leq (\frac{1}{p^n})^{d- \text{dim}(M)}C_1e^{CC'}.$$
for some constant $C_1$. 
\end{proof}
\noindent The next few results are aimed towards \Cref{reduction to domain}.
\begin{lemma}\label{vanishing}
Let 
\begin{tikzcd}
0 \arrow[r] & K \arrow[r, "\phi_1"] & M_1 \arrow[r, "\phi"] & M_2 \arrow[r, "\phi_2"] & C \arrow[r] & 0
\end{tikzcd}
be an exact sequence of finitely generated $\Z$-graded $R$-modules (i.e. assume the boundary maps preserve the respective gradings). Let $d$ be an integer greater than both $\text{dim}(K) \, \text{and} \, \, \text{dim}(C)$. 
\begin{enumerate}
   \item Given a compact subset $A\subseteq \mathbb{C}$, there exists a constant $D$, such that for all $y \in A$ and $n \in \N$, 
   $$|F_n(M_2,R,I,d)(y)- F_n(M_1, R,I, d)(y)| \leq \frac{D}{p^n}.$$
   \item The sequence of functions $(F_n(M_2,R,I,d)- F_n(M_1, R,I, d))_n$ converges to the constant function zero and the convergence is uniform on every compact subset of $\mathbb{C}$.
\end{enumerate}
\end{lemma}
\begin{proof}
We prove assertion (1) below, assertion (2) is immediate from assertion (1).\\

Break the given exact sequence into two short exact sequences:
\begin{equation}
    \begin{tikzcd}
    0 \arrow[r] & K \arrow[r, "\phi_1"] & M_1 \arrow[r, "\phi"] & Im(\phi) \arrow[r] & 0 \label{*} \tag{$*$}
    \end{tikzcd} \ ,
\end{equation}
\begin{equation}
    \begin{tikzcd}
    0 \arrow[r] & Im(\phi) \arrow[r] & M_2 \arrow[r, "\phi_2"] & C \arrow[r] & 0 \label{**} \tag{$**$} \ .
    \end{tikzcd}
\end{equation}
Now apply $\otimes_R\frac{R}{\bp{I}{n}R}$ to \eqref{*} and \eqref{**}; the corresponding long exact sequences of Tor modules give the following two exact sequences of graded modules for each $n$:
\begin{equation}\label{*_n}
\begin{tikzcd}
\frac{K}{\bp{I}{n}K} \arrow[r, "{\phi_{1,n}}"] & \frac{M_1}{\bp{I}{n}M_1} \arrow[r] & \frac{Im(\phi)}{\bp{I}{n}Im(\phi)} \arrow[r] & 0 \tag{$*_n$} \ ,
\end{tikzcd}
\end{equation}
\begin{equation}\label{**_n}
\begin{tikzcd}
\text{Tor}^R_1(C, \frac{R}{\bp{I}{n}R})\arrow[r, "\tau_n"]& \frac{Im(\phi)}{\bp{I}{n}Im(\phi)} \arrow[r] & \frac{M_2}{\bp{I}{n}M_2} \arrow[r] & \frac{C}{\bp{I}{n}C} \arrow[r] & 0 \tag{$**_n$} \ .
\end{tikzcd}
\end{equation}
Using \eqref{*_n} and \eqref{**_n}, for each $j \in \Z$, we get
$$\lambda \left((\frac{M_1}{\bp{I}{n}M_1})_j\right) = \lambda \left((\phi_{1,n}(\frac{K}{\bp{I}{n}K}))_j\right)+ \lambda \left((\frac{Im(\phi)}{\bp{I}{n}Im(\phi)})_j \right),$$
$$\lambda \left((\frac{M_2}{\bp{I}{n}M_2})_j\right)= \lambda\left((\frac{Im(\phi)}{\bp{I}{n}Im(\phi)})_j \right) - \lambda \left(\tau_n(\text{Tor}^R_1(C, \frac{R}{\bp{I}{n}R}))_j\right)+ \lambda \left((\frac{C}{\bp{I}{n}C})_j\right).$$
Therefore, 
\begin{equation}\label{eq1}
\begin{split}
& F_n(M_2,R,I,d)(y)- F_n(M_1, R,I, d)(y)\\
& = F_n(C,R,I,d)(y) -\expand{\tau_n(\text{Tor}^R_1(C, \frac{R}{\bp{I}{n}R})}{d}\\
&-\expand{\phi_{1,n}(\frac{K}{\bp{I}{n}K})}{d} \ .
\end{split}
\end{equation}
By \Cref{support}, one can choose a positive integer $C_1$ such that given any $n$ and all $m$ such that $|m| >C_1p^n$,
$$(\frac{C}{\bp{I}{n}C})_m= (\text{Tor}^R_1(C, \frac{R}{\bp{I}{n}R}))_m= (\frac{K}{\bp{I}{n}K})_m=0 .$$
Since $A$ is compact, there is a constant $C_2$ such that for all $j$, where $|j| \leq C_1p^n$ and for $y \in A$, $|e^{-iyj/p^n}| \leq C_2$- the argument is similar to that in \eqref{bound on the sequence}.
Using \eqref{eq1}, we conclude that  for $y \in A$,
\begin{equation*}
\begin{split}
  & |F_n(M_2,R,I,d)(y)- F_n(M_1, R,I, d)(y)|\\
  &\leq (\frac{1}{p^n})^d\underset{|j| \leq C_1p^n}{\sum}[\lambda((\frac{C}{\bp{I}{n}C})_j)+ \lambda((\text{Tor}^R_1(C, \frac{R}{\bp{I}{n}R}))_j)+ \lambda((\frac{K}{\bp{I}{n}K})_j)]|e^{-iyj/p^n}|\\
  &\leq C_2(\frac{1}{p^n})^d[\lambda(\frac{C}{\bp{I}{n}C})+ \lambda(\text{Tor}^R_1(C, \frac{R}{\bp{I}{n}R}))+ \lambda(\frac{K}{\bp{I}{n}K})] \ .
\end{split}
\end{equation*}
Since both $\text{dim}(C)$ and $\text{dim}(K)$ are less than $d$, the desired result follows from \Cref{bounded growth}.
\end{proof}
\vspace{.5cm}
\noindent Recall that for an integer $h$, $M(h)$ denotes the $R$-module $M$ but with a different $\Z$-grading: the $n$-th graded piece of $M(h)$ is $M_{n+h}$. From now on, we use the terminology set in the next definition.
\begin{definition}
Whenever the sequence of complex numbers $(F_n(M,R,I,d)(y))_n$ (see \eqref{the sequence}) converges, we set
 $$F(M,R,I,d)(y)= \underset{n \to \infty}{\lim}F_n(M,R,I,d)(y).$$
In the case $d= \text{dim}(M)$, we set $F(M,R,I)(y)= F(M,R,I,\text{dim}(M))(y)$. Analogously we use $F(M)(y)$ when $R,I$ are clear from the context.
\end{definition}
\begin{proposition}\label{invariance after a shift}
Let $R$ be a finitely generated $\N$-graded $k$-algebra, $I$ be a homogeneous ideal of finite co-length and $M$ be a finitely generated $\Z$-graded $R$-module. Fix an integer $h$.
\begin{enumerate}
    \item Given a compact subset $A\subseteq \mathbb{C}$, there exists a constant $D$ such that for all $y \in A$ and all $n$,
    $$|F_n(M(h), R,I,d)(y)-F_n(M,R,I,d)(y)| \leq \frac{D}{p^n}|F_n(M,R,I,d)(y)| .$$
   \item For any complex number $y$, $(F_n(M,R,I,d)(y))_n$ converges if and only if\\ $(F_n(M(h), R,I,d)(y))_n$ converges. When both of these converge, their limits are equal.
\end{enumerate}\end{proposition}
\begin{proof}
(1) Note that
$$F_n(M(h),R,I,d)(y)=(\frac{1}{p^n})^d \sum \limits_{j=- \infty}^{\infty}\lambda((\frac{M}{\bp{I}{n}M})_{j+h})e^{-iy(j+h)/p^n}e^{iyh/p^n}= e^{iyh/p^n}F_n(M,R,I,d)(y).$$
Thus,
$$|F_n(M(h), R,I,d)(y)-F_n(M,R,I,d)(y)|= |e^{iyh/p^n}-1||F_n(M,R,I,d)(y)|.$$
Since $A$ is bounded, it follows from \Cref{the essential limit}, there is a constant $D$ such that for $y \in A$, $|1-e^{-ihy/p^n}| \leq \frac{D}{p^n}$.\\\\
(2) It follows from the first assertion that whenever $(F_n(M,d)(y))_n$ converges, the sequence $F_n(M(h), d)(y)$ also converges and the two limits coincide. The other direction follows from the observation that as a graded module $M$ is isomorphic to $(M(h))(-h)$.
\end{proof}
\begin{lemma}\label{restriction of scalars}
Let $R \rightarrow S$ be a degree preserving finite  homomorphism of finitely generated $\N$-graded $k$-algebras. For any finitely generated $\Z$-graded $S$-module $N$ and any complex number $y$, $(F_n(N, R,I,d)(y))_n$ converges if and only if $(F_n(N, S, IS, d)(y))_n$ converges. When both of these converge $F(N,R,I,d)(y)= F(N,S,IS,d)(y)$.
\end{lemma}
\begin{proof}
Since the $R$-module structure on $N$ comes via the restriction of scalars, for each $n$, the two $k$-vector spaces $(\frac{N}{\bp{I}{n}N})_j$ and $(\frac{N}{\bp{(IS)}{n}N})_j$ are isomorphic. Thus $F_n(N,R,I,d)(y)= F_n(N, S, IS, d)(y)$ and the conclusion follows.
\end{proof}
Note that, for a finitely generated $\N$-graded $k$-algebra $R$, $R^{p^e}= \{r^{p^e} \, | \, r \in R\}$ is an $\N$-graded subring of $R$- the $\N$-grading on $R^{p^e}$ will refer to this grading.

\begin{proposition} \label{base change}
Let $k$ be a field of characteristic $p>0$ such that  $k^p \subseteq k$  is finite. Let $R$ be a finitely generated $\N$-graded $k$-algebra, $I$ be a homogeneous ideal of finite co-length and  $M$ be a finitely generated $\Z$-graded $R$-module. Given two non-negative integers $d,m$ and a complex number $y$, 

\begin{enumerate}
    \item Denote the image of $I$ in $R^{p^m}$ under the $p^m$-th power map by $IR^{p^m}$. Then
   \begin{equation*}
        \begin{split}
             F_{n+m}(M,R,I,d)(y)&= \frac{1}{p^{md}[k:k^{p^m}]} F_n(M,R^{p^m},IR^{p^m},d)(y/p^m) \ .
        \end{split}
    \end{equation*}

    \item When $(F_n(M,R,I,d)(y))_n$ converges, $$F(M,R,I,d)(y)= \frac{1}{p^{md}[k:k^{p^m}]}F(M, R^{p^m}, IR^{p^m}, d)(y/p^m).$$
    
    \item If $R$ is reduced, for all $n$ $$F_n(R^p, R^p, IR^p,d)(y/p)= F_n(R,R,I,d)(y).$$ 
\end{enumerate}
\end{proposition}
\begin{proof}
Given $n \in \N$,

\begin{equation*}
\begin{split}
F_n(M,R^{p^m},IR^{p^m},d)(y/p^m) & =(\frac{1}{p^n})^d \sum \limits_{j=-\infty}^{\infty}\lambda_{k^{p^m}}((\frac{M}{\bp{(IR^{p^m})}{n}M})_j)e^{-iyj/p^{n+m}}\\
& =(\frac{1}{p^n})^d \sum \limits_{j=-\infty}^{\infty}\lambda_{k^{p^m}}((\frac{M}{\bp{I}{n+m}M})_j)e^{-iyj/p^{n+m}}\\
& = p^{md}[k:k^{p^m}] (\frac{1}{p^{n+m}})^d \sum \limits_{j=-\infty}^{\infty}\lambda_k((\frac{M}{\bp{I}{n+m}M})_j)e^{-iyj/p^{n+m}}\\
& = p^{md}[k:k^{p^m}]F_{n+m}(M,R,I,d)(y) \ .
\end{split}
\end{equation*}

(1) and (2) follows directly from the calculation  above.\\\\
Now, we verify (3). Since $R$ is reduced, the Frobenius $F_R: R \rightarrow R$ induces an isomorphism onto 
$R^p$; it takes $R_j$ to $(R^p)_{jp}$. Thus for each $n,j \in \N$, $F_R$ induces an isomorphism of abelian groups from $(\frac{R}{\bp{I}{n}})_j$ to $(\frac{R^p}{\bp{(IR^p)}{n}R^p})_{jp}$.
So, $\lambda_k((\frac{R}{\bp{I}{n}})_j)= \lambda_{k^p}((\frac{R^p}{\bp{(IR^p)}{n}R^p})_{jp})$. Now,

\begin{equation*}
\begin{split}
    F_n(R^p, R^p, IR^p,d)(y/p)
    &= (\frac{1}{p^n})^d\sum \limits_{j=0}^{\infty}\lambda_{k^p}((\frac{R^p}{\bp{(IR^p)}{n}R^p})_{jp})e^{-iyj/p^n}\\
    &= (\frac{1}{p^n})^d\sum \limits_{j=0}^{\infty}\lambda_k((\frac{R}{\bp{I}{n}R})_j)e^{-iyj/p^n}.
\end{split}
\end{equation*}

\noindent The rightmost quantity on the above equality is $F_n(R,R,IR,d)(y)$.
\end{proof}
\begin{theorem}\label{reduction to domain}
Let $R$ be a finitely generated $\N$-graded $k$-algebra, $I$ be a homogeneous ideal of finite co-length. Let $M$ be a finitely generated $\Z$-graded $R$-module of Krull dimension $d$ and $Q_1, \ldots, Q_l$ be the $d$ dimensional minimal prime ideals in the support of $M$. Let $m$ be such that $\bp{\text{nil}(R)}{m}$ is zero, where $\text{nil}(R)$ is the nilradical of $R$.

\begin{enumerate}
    \item Given 
a compact subset $A \subseteq \mathbb C$, there exists a constant $D$ such that for all $y \in A$, 

\begin{equation*}
    |F_{n+m}(M,R,I,d)(y)- \sum \limits_{j=1}^{l}\lambda_{R_{Q_j}}(M_{Q_j})F_n(\frac{R}{Q_j}, \frac{R}{Q_j}, I\frac{R}{Q_j},d)(y)| \leq \frac{D}{p^n} \ .
\end{equation*}

    \item Given $y \in \mathbb C$, whenever $(F_n(\frac{R}{Q_j}, \frac{R}{Q_j}, I\frac{R}{Q_j},d)(y))_n$ is convergent for every $j$, $(F_n(M, R,I,d)(y))_n$ is also convergent and

    \begin{equation}\label{the equality in localization formula}
F(M,R,I,d)(y)= \sum \limits_{j=1}^{l}\lambda_{R_{Q_j}}(M_{Q_j})F(\frac{R}{Q_j}, \frac{R}{Q_j}, I\frac{R}{Q_j},d)(y) \ .
\end{equation}
\end{enumerate}
\end{theorem}

To prove \ref{reduction to domain}, we first establish several lemmas to handle the reduced case. The main algebraic input into the proof is the following.
\begin{lemma}\label{for one prime}
Let $R$ be a reduced finitely generated $\N$-graded $k$-algebra and let $Q$ be a minimal prime ideal of $R$. Let $U$ be the multiplicative set of homogeneous elements in $R-Q$. Then,
\begin{enumerate}
    \item The ideal $QU^{-1}R$ is zero. Moreover, there is a field $k$ such that $U^{-1}R$ is isomorphic to either $k$ or $k[t,t^{-1}]$, where $t$ is an indeterminate over $k$. 
    \item Set $r= \lambda_{R_Q}(N_Q)$. Then there exist integers $h_1, \ldots, h_r$ and a grading preserving $R$-linear morphism
    $$\phi_Q: \bigoplus \limits_{j=1}^{r}\frac{R}{Q}(-h_j) \longrightarrow N,$$
    such that the map induced by $\phi_Q$ after localizing at $Q$ is an isomorphism.
\end{enumerate}
\end{lemma}
\begin{proof}
(1)Any non-zero homogeneous prime ideal of $U^{-1}R$ is the extension of a homogeneous prime ideal of $R$ contained in $R\setminus U$; so is contained in $Q$. As $Q$ is minimal, we conclude that $U^{-1}R$ has a unique prime ideal namely $QU^{-1}R$. Since $R$ is reduced, so is $U^{-1}R$. So $QU^{-1}R=0$. Since $U^{-1}R$ does not have any non-zero homogeneous prime ideal, every non-zero homogeneous element of $R$ is a unit. Therefore $U^{-1}R$ is isomorphic to either $k$ or $k[t, t^{-1}]$ for some field $k$; see \cite[Lemma 1.5.7]{BH}.\\\\
(2) Since $R$ is reduced and $Q$ is a minimal prime, $R_Q$ is a field. We produce $r$ homogeneous elements of $N$, each of which is annihilated by $Q$ and their images in $N_Q$ form an $R_Q$-basis of $N_Q$. For that, start with $r$ homogeneous elements $m_1', \ldots, m_r'$ such that $\{\frac{m_1'}{1}, \ldots, \frac{m_r'}{1}\}$ is an $R_Q$-basis of $N_Q$. Since by part (1) $QU^{-1}R$ is the zero ideal and $Q$ is finitely generated, we can pick an element $s$ in $U$ such that $s$ annihilates $Q$. Now set $m_j= sm_j'$ for each $j$. Each $m_j$ is annihilated by $Q$. Since $s$ is not in $Q$, the images of $m_1, \ldots, m_r$ in $N_Q$ form an $R_Q$-basis of $N_Q$. \\

Now, set $h_j= \text{deg}(m_j)$. Let 
$$\phi_Q: \bigoplus \limits_{j=1}^{r}\frac{R}{Q}(-h_j) \longrightarrow N$$
be the $R$-linear map sending $1 \in \frac{R}{Q}(-h_j)$ to $m_j$. Clearly $\phi_Q$ preserves gradings. Since the images of $m_1, \ldots, m_r$ form an $R_Q$-basis of $M_Q$, the map induced by $\phi_Q$ after localizing at $Q$ is an isomorphism, so our desired conclusion in \Cref{generic behaviour} follows.
\end{proof}

\begin{lemma}\label{generic behaviour}
Suppose that $R$ is reduced and let $P_1, P_2, \ldots, P_t$ be those among the minimal prime ideals of $R$ such that $\text{dim}(R) = \text{dim}(\frac{R}{P_j})$. Let $N$ be a finitely generated $\Z$-graded $R$-module. For each $j$, where $1 \leq j \leq t$, let $r_j = \lambda_{R_{P_j}}(N_{P_j})$. Then there exist integers $h_{j,n_j}$ where $1 \leq j \leq t, 1 \leq n_j \leq r_j$ and a degree preserving $R$-linear map,
$$\phi: \bigoplus \limits_{j=1}^{t} \bigoplus \limits_{n_j=1}^{r_j}\frac{R}{P_j}(-h_{j,n_j}) \longrightarrow N ,$$
such that the $\text{dim}_R(\text{ker}(\phi)) < \text{dim}(R)$, $\text{dim}_R(\text{coker}(\phi))<\text{dim}(R)$.
\end{lemma}

\begin{proof}
Consider for each $j$, $1 \leq j \leq t$, a $\phi_{P_j}$ as in assertion (2) of \Cref{for one prime}. Let $\phi$ be the map induced by these $\phi_{P_j}$'s. Since $P_1, \ldots, P_t$ are all distinct \textit{minimal} primes, after localizing at any $P_j$, the maps induced by $\phi$ and $\phi_{P_j}$ coincide. So the map induced by $\phi$ after localizing at each $P_j$ is an isomorphism. Hence, none of the supports of kernel and cokernel of $\phi$ include any of $P_1, \ldots, P_t$. Since $P_1, \ldots, P_t$ are precisely the minimal primes of $R$ of maximal dimension, Lemma \ref{generic behaviour} is proved.
\end{proof}

\begin{proof}[\textbf{Proof of Theorem \ref{reduction to domain}:}]
The second assertion follows from the first one; so we just prove the first assertion below.

 By \Cref{reduction to the F-finite case} we can assume that $k^p \subseteq k$ is a finite extension. Using \Cref{restriction of scalars} we can replace $(M,R,I)$ by $(M, \frac{R}{\text{Ann}(M)}, I\frac{R}{\text{Ann}(M)})$. So we assume that $d= \text{dim}(R)$.
 
 First we additionally assume that $R$ is reduced and show that taking $m=0$ works in assertion (1). By \Cref{restriction of scalars}, for all $j$ and $n$, $$F_n(\frac{R}{Q_j}, \frac{R}{Q_j}, I\frac{R}{Q_j}, d)(y)= F_n(\frac{R}{Q_j}, R, I, d)(y).$$
 Assertion (1) follows from direct applications of \Cref{generic behaviour}, assertion (1) of \Cref{vanishing} and assertion (1) of \Cref{invariance after a shift}. \\
 
We now prove assertion (1) of \Cref{reduction to domain} without assuming $R$ is reduced. We use the Frobenius endomorphism to pass to the reduced case. Pick an $m$ such that $\bp{\text{nil}(R)}{m}=0$. Then the kernel of the $m$-th iteration of the Frobenius $F_R^m: R \rightarrow R$ is $\text{nil}(R)$; thus $R^{p^m}$- the image of $F_R^m$- is reduced. Recall $R^{p^m}$ inherits the graded structure of $R$. The $d$-dimensional minimal primes of $R^{p^m}$ in the support of the $R^{p^m}$ module $M$ are precisely of $Q_1R^{p^m}, \ldots, Q_lR^{p^m}$ the respective images under the $p^m$-th power map. Since $R^{p^m}$ is reduced and $\frac{1}{p^m}A:= \{z/p^m \, | \, z \in A\}$ is compact, we can find a $D$ such that for each $y \in A$ and all $n$,

\begin{gather}\label{ga: after passing to the image via Frobenius iteration}
  | F_n(M,R^{p^m}, IR^{p^m})(y/p^m)
  - \sum \limits_{j=1}^l \lambda_{R^{p^m}_{Q_jR^{p^m}}}\left((M)_{Q_jR^{p^m}}\right)F_n(\frac{R^{p^m}}{Q_jR^{p^m}}, \frac{R^{p^m}}{Q_jR^{p^m}}, I\frac{R^{p^m}}{Q_jR^{p^m}})(y/p^m)|\\
  \leq \frac{D}{p^n} \ .  
\end{gather}

\noindent For each $j$, the graded ring $\frac{R^{p^m}}{Q_jR^{p^m}}$ is isomorphic to the graded subring $(\frac{R}{Q_j})^{p^m} \subset \frac{R}{Q_j}$; so for all $n$ and $y \in \mathbb{C}$,

\begin{equation*}
  F_n((\frac{R}{Q_j})^{p^m}, (\frac{R}{Q_j})^{p^m}, I(\frac{R}{Q_j})^{p^m}, d)(y)= F_n(\frac{R^{p^m}}{Q_jR^{p^m}}, \frac{R^{p^m}}{Q_jR^{p^m}}, I\frac{R^{p^m}}{Q_jR^{p^m}}, d)(y) \ .
\end{equation*}

Since $\frac{R}{Q_j}$ is reduced,
$$F_n((\frac{R}{Q_j})^{p^m}, (\frac{R}{Q_j})^{p^m}, I(\frac{R}{Q_j})^{p^m}, d)(y/p^m)= F_n(\frac{R}{Q_j}, \frac{R}{Q_j}, I\frac{R}{Q_j}, d)(y)$$
by \Cref{base change}, (3).

Using \Cref{base change}, assertion (1), we have

\begin{equation*}
   p^{md}[k:k^{p^m}] F_{n+m}(M,R,I,d)(y) = F_n(M,R^{p^m},IR^{p^m},d)(y/p^m) \ .
\end{equation*}

Since for each $j$, $\frac{R}{Q_j}$ has krull dimension $d$, 
$$\lambda_{R^{p^m}_{Q_jR^{p^m}}}((M)_{Q_jR^{p^m}})= [k:k^{p^m}]p^{dm} \lambda_{R_{Q_j}}(M_{Q_j})$$

So \Cref{ga: after passing to the image via Frobenius iteration} yields

$$|F_{n+m}(M,R,I,d)(y)- \sum \limits_{j=1}^{l}\lambda_{R_{Q_j}}(M_{Q_j})F_n(\frac{R}{Q_j}, \frac{R}{Q_j}, I\frac{R}{Q_j},d)(y)| \leq \frac{D}{p^{dm}[k:k^{p^m}]}\frac{1}{p^n},$$

proving assertion (1).

\end{proof}
\begin{proof}[\textbf{Proof of \Cref{the main theorem}:}] Using \Cref{reduction to the F-finite case} we can assume that $k^p \subseteq k$ is a finite extension. We argue that the sequence $(F_n(M,R,I, \text{dim}(M)))_n$ is uniformly Cauchy on every compact subset. By \Cref{reduction to domain}, assertion (1), we can assume that $R$ is a domain and $M=R$. Fix a compact subset $A \subseteq \mathbb{C}$. Since the torsion free rank of $R$ as an $R^p$ module is $p^d[k:k^p]$, we have an exact sequence of finitely generated graded $R^p$ modules (see \Cref{generic behaviour}):
$$\begin{tikzcd}
0 \arrow[r] & K \arrow[r] & \bigoplus\limits_{j=1}^{p^d[k:k^p]}R^p(h_j) \arrow[r] & R \arrow[r] & C \arrow[r] & 0
\end{tikzcd}$$
for some integers $h_i$ such that both $\text{dim}_{R^p}(K) \, \text{and} \, \text{dim}_{R^p}(C)$ are less than $d$. Hence there exist constants $D, D'$ such that for all $n$ and for any $y \in A$,
\begin{equation*}
\begin{split}
|F_{n+1}(R,R,I)(y)-F_n(R,R,I)(y)|
    & =|\frac{1}{p^d[k:k^p]}F_n(R,R^p,IR^p)(y/p)- F_n(R,R,I)(y)|\\
    & \leq |\frac{1}{p^d[k:k^p]}\sum \limits_{j=1}^{p^d[k:k^p]}F_n(R^p(h_j),R^p,IR^p)(y/p)\\
    &-F_n(R,R,I)(y)|+ \frac{D}{p^n}\\
    & \leq |F_n(R^p,R^p, IR^p)(y/p)-F_n(R,R,I)(y)|+\frac{D'}{p^n}+\frac{D}{p^n}\\
    & = \frac{D+D'}{p^n} \ .
\end{split}
\end{equation*}
The first equality comes from assertion 1 of \Cref{base change}. The first inequality is a consequence of assertion (1) of \Cref{vanishing}. The second inequality is obtained by applying assertion (1) of \Cref{invariance after a shift} and assertion (1) of \Cref{FP functions are bounded on a compact subset}. The last equality follows from \Cref{base change}, assertion (3).
Hence for $m,n \in \N$ and for any $y \in A$, $$|F_{n+m}(R,R,I)(y)-F_n(R,R,I)(y)| \leq (D+D')(\sum \limits_{j=n}^{\infty}\frac{1}{p^j})= \frac{D+D'}{p^n}\frac{p}{p-1} .$$
Thus the sequence of entire functions $(F_n(R,R,I)(y))_n$ is uniformly Cauchy on $A$.\\

A sequence of entire functions which is uniformly Cauchy on every compact subset of $\mathbb C$ converges to a entire function and the convergence is uniform on every compact subset; see Theorem 1 in Chapter 5 of \cite{Ahlfors}. This finishes the proof of \Cref{the main theorem}.
\end{proof} 
\section{Properties of Frobenius-Poincar\'e functions}
\noindent This section is devoted to developing general properties of Frobenius-Poincar\'e functions. Some of these are analogues of properties of Hilbert-Kunz multiplicities. In \Cref{FP function wrt a homogeneous SOP} and \Cref{FP function in dimension one}, we use these general properties to compute Frobenius-Poincar\'e functions in some special cases.
\begin{proposition}
 \label{power series expansion}
Let $M$ be a finitely generated $\Z$-graded $R$-module of Krull dimension $d$. Then the power series expansion of $F(M,R,I)(y)$ around the origin in the complex plane is given by
$$F(M,R,I)(y)= \sum \limits_{m= 0}^{\infty} a_m y^m ,$$
where for each $m$, 
$$a_m= (-i)^m\frac{1}{m!} \, \underset{n \to \infty}{\lim}(\frac{1}{p^n})^{d+m}\sum \limits_{j=- \infty}^{\infty}j^m \lambda{((\frac{M}{\bp{I}{n}M})_j)} .$$
\end{proposition}
\begin{proof}
Since the sequence $(F_n(M))_n$ converges uniformly to $F(M)$ on the closed unit disc around zero, it follows from Lemma 3, Chapter 4 of \cite{Ahlfors} that  for each $m$, the sequence $$\frac{d^m}{dy^m}(F_n)(0)=(-i)^m (\frac{1}{p^n})^{d+m}\sum \limits_{j= -\infty}^{\infty}j^m \lambda{((\frac{M}{\bp{I}{n}M})_j)}$$ converges to $\frac{d^m}{dy^m}(F)(0)$. Since $a_m= \frac{1}{m!}\frac{d^m}{dy^m}(F)(0)$, we get the result.
\end{proof}
\begin{corollary}\label{Hk multiplicity from the Frobenis-Poincare function}
The Hilbert-Kunz multiplicity of the triple $(M, R,I)$ is $F(M,R,I)(0)$.
\end{corollary}
\vspace{.5cm}
\noindent The next result provides a associativity formula for Frobenius-Poincar\'e functions.
\begin{theorem}\label{FP functions are described by generic ranks}
Let $M$ be a finitely generated $\Z$-graded $R$-module of Krull dimension $d$. Let $P_1, \ldots, P_t$ be the dimension $d$ minimal prime ideals in the support of $M$. Then
$$F(M,R,I,d)(y)= \sum \limits_{j=1}^{t}\lambda_{R_{P_j}}(M_{P_j})F(\frac{R}{P_j}, \frac{R}{P_j}, I\frac{R}{P_j},d)(y).$$
\end{theorem}
\begin{proof}
Follows from \Cref{reduction to domain}.
\end{proof}
As a consequence of \Cref{FP functions are described by generic ranks}, we prove that Frobenius-Poincar\'e functions are additive over a short exact sequence.
\begin{proposition}\label{additivity over short exact sequences}
Consider a short exact sequence of finitely generated $\Z$-graded $R$-modules where the boundary maps preserve gradings,
     $$\begin{tikzcd}
    0 \arrow[r] & M' \arrow[r] & M \arrow[r] & M'' \arrow[r] & 0 .
    \end{tikzcd}
$$ 
    Let $d$ be the Krull dimension of $M$. Then $F(M,R,I,d)= F(M',R,I,d)+ F(M'', R,I,d)$.
\end{proposition}
\begin{proof} The support of $M$ is the union of supports of $M'$ and $M''$. Since for a $d$ dimensional minimal prime $Q$ in the support of $M$, $\lambda_{R_Q}(M_Q)= \lambda_{R_Q}(M'_Q)+ \lambda_{R_Q}(M''_Q)$, the desired result follows from \Cref{FP functions are described by generic ranks}.
\end{proof}
In \Cref{FP function wrt a homogeneous SOP} we apply \Cref{FP functions are described by generic ranks} to compute the Frobenius-Poincar\'e function with respect to an ideal generated by a homogeneous system of parameters.
\begin{proposition}\label{FP function wrt a homogeneous SOP}
Let $R$ be an $\N$-graded, Noetherian ring such that $R_0=k$. Let $I$ be an ideal generated by a homogeneous system of parameters of degrees $\delta_1, \delta_2, \ldots, \delta_d$. Denote the Hilbert-Samuel multiplicity of $R$ by $e_R$- see \Cref{HS multiplicity}. Then
$$F(R,I)(y)= e_R\prod \limits_{j=1}^{d}(\frac{1-e^{-i \delta_j y}}{i y}).$$
\end{proposition}
\begin{proof}
Suppose $f_1, \ldots, f_d$ be a homogeneous system of parameters of degrees $\delta_1, \ldots, \delta_d$ respectively, such that $I= (f_1, \ldots, f_d)$. Then the extension of rings $k[\underline{f}]:= k[f_1, \ldots, f_d] \hookrightarrow R$ is finite (see Theorem 1.5.17, \cite{BH}). Suppose that the generic rank of $R$ as an $k[\underline{f}]$ module is $r$. Since $k[\underline{f}]$ is isomorphic to the graded polynomial ring in $d$ variables where the degrees of the variables are $\delta_1, \ldots, \delta_d$, from \Cref{multivariable polynomial ring} and \Cref{FP functions are described by generic ranks}, we have
\begin{equation}\label{before finding r}
F(R,I)(y)= F(R, k[\underline{f}], (f_1, \ldots, f_d)k[\underline{f}])= r\prod \limits_{j=1}^{d}(\frac{1-e^{-i \delta_j y}}{i \delta_j y}) \ .
\end{equation}
Taking limit as $y$ tends to zero in \eqref{before finding r} and \cref{power series expansion}, we conclude that $r$ is the Hilbert-Kunz multiplicity of the pair $(R,I)$. The Hilbert-Kunz and the Hilbert-Samuel multiplicities are the same with respect to a given ideal generated by a system of parameters (see Theorem 11.2.10, \cite{HunekeSwanson}). So using \Cref{HS multiplicity with respect to SOP} we get that $r= \delta_1 \ldots \delta_j e_R$.
\end{proof}
Now we compute the Frobenius-Poincar\'e function of a one dimensional graded domain whose degree zero piece is an algebraically closed field. This will indeed allow us to compute the Frobenius-Poincar\'e function of any one dimensional graded ring by using \Cref{reduction to the F-finite case} and \Cref{FP functions are described by generic ranks}.
\begin{proposition}\label{FP function in dimension one}
Let $R$ be a one dimensional finitely generated $\N$-graded $k$-algebra , where $k$ is algebraically closed and $R$ is a domain. Let $I$ be a finite co-length homogeneous ideal. Let $h$ be the smallest integer such that $I$ contains a non-zero homogeneous element of degree $h$. Then $$F(R,R,I)(y)= e_R(\frac{1-e^{-ihy}}{iy}),$$ where $e_R$ is the Hilbert-Samuel multiplicity of $R$ (see \Cref{HS multiplicity}).
\end{proposition}
\begin{proof}
Let $S$ be the normalization of $R$. By Theorem 11, chapter VII, \cite{ZS2}, $S$ is an $\N$-graded $R$-module and by Theorem 9, chapter 5, \cite{ZS1} $S$ is finitely generated over $k$. The generic rank of $S$ as an $R$-module is one, hence $F(S,R,I)= F(R,R,I)$. Since $k$ is algebraically closed $S_0= k$. So by \Cref{restriction of scalars} $F(S,R,I)$ is the same as $F(S,S,IS)$. So we compute $F(S,S, IS)$. Since $S$ is an $\N$-graded normal $k$-algebra, by Theorem 1, section 3, Appendix III of \cite{Serre}, $S$ is isomorphic to a graded polynomial ring in one variable. So the ideal $IS$ is a homogeneous principal ideal. By our assumption, $IS$ is generated by a degree $h$ homogeneous element $f$; note that $f$ is a homogeneous system of parameter of $S$. Thus by \Cref{FP function wrt a homogeneous SOP}, $F(S,S, IS)(y)= e_S(\frac{1-e^{-ihy}}{iy})$. Since $S$ has generic rank one as an $R$ module $e_R=e_S$.
\end{proof}
Let $S$ be a ring containing a field of characteristic $p>0$. Recall that the tight closure of an ideal $J \subseteq S$ is the ideal consisting of all $x \in S$ such that there is a $c \in S$, not in any minimal prime of $S$ such that, $cx^{p^n} \in \bp{J}{n}$ for all large $n$ --see Definition 3.1, \cite{HH}). The tight closure of $J$ is denoted by $J^*$. The theory of Hilbert-Kunz multiplicity is related to the theory of tight closure: for a Noetherian local domain $S$ whose completion is also a domain and ideals $J_1 \subseteq J_2$, the corresponding Hilbert-Kunz multiplicities are the same if and only if $J_1^*= J_2^*$ - see Proposition 5.4, Theorem 5.5 of \cite{HunekeExp} and Theorem 8.17, \cite{HH}. A similar relation between tight closure of an ideal and the corresponding Frobenius-Poincar\'e function is the content of the next result.
\begin{theorem}\label{dependence on tight closure}
Let $R$ is a finitely generated $\N$- graded $k$-algebra. Let $I \subseteq J$ be two finite colength homogeneous ideals of $R$
\begin{enumerate}
    \item If $J$ is contained in $I^*$-the tight closure of $I$, $F(R,R,I)= F(R,R,J)$.
    \item Suppose that all of the minimal primes of $R$ have the same dimension. If $F(R,R,I)= F(R,R,J)$, $J \subseteq I^*$.
\end{enumerate}
\end{theorem}
\begin{proof}
1) Denote the Krull dimension of $R$ by $d$. For (1), first we argue that there is a constant $D$ such that  $\lambda(\frac{\bp{J}{n}}{\bp{I}{n}})$ is bounded above by $D(p^n)^{d-1}$ for all large n. Since $J \subseteq I^*$, there exists a $c \in R$- not in any minimal primes of $R$ such that $c \bp{J}{n} \subseteq \bp{I}{n}$, for all large $n$\footnote{$c$ can be chosen to be homogeneous}. Pick a set of homogeneous generators of $g_1, g_2, \ldots, g_{r}$ of $J$. Since the images of $g_1^{p^n},\ldots, g_r^{p^n}$ generate $\frac{\bp{J}{n}}{\bp{I}{n}}$, we get a surjection for each $n$:
$$\bigoplus \limits_{j=1}^{r}\frac{R}{(c, \bp{I}{n})} \twoheadrightarrow \frac{\bp{J}{n}}{\bp{I}{n}}.$$
So the length of $\frac{\bp{J}{n}}{\bp{I}{n}}$ is bounded above by $r \lambda(\frac{R}{(c, \bp{I}{n})})$. Since $c$ is not in any minimal prime of $R$, $\text{dim}(\frac{R}{cR})$ is at most $d-1$. The existence of the desired $D$ is apparent once we use \Cref{bounded growth} to bound the growth of $\lambda(\frac{R}{(c, \bp{I}{n})})$.\\\\
Now, pick $N_0 \in \mathbb{N}$ such that for $j \geq N_0 p^n$, $(\frac{R}{\bp{I}{n}R})_j=0$ for all $n$. Given $y \in \mathbb{C}$,
\begin{equation*}
    \begin{split}
      |F_n(R,R,I)(y)- F_n(R,R,J)(y)|
      & = (\frac{1}{p^n})^d |\sum \limits_{j=0}^{\infty}\lambda((\frac{\bp{J}{n}}{\bp{I}{n}})_j)e^{-iyj/p^n}|\\
      & \leq (\frac{1}{p^n})^d\sum \limits_{j=0}^{\infty}|\lambda((\frac{\bp{J}{n}}{\bp{I}{n}})_j)||e^{-iyj/p^n}|\\
      & \leq (\frac{1}{p^n})^d\lambda(\frac{\bp{J}{n}}{\bp{I}{n}})e^{N_0|y|} \ .
    \end{split} 
\end{equation*}
To get the last inequality, we have used that for $j \leq N_0p^n$, $|e^{-iyj/p^n}| \leq e^{N_0|y|}$. Since $$\underset {n \to \infty}{\lim}(\frac{1}{p^n})^d \lambda(\frac{\bp{J}{n}}{\bp{I}{n}}) \leq \underset{n \to \infty}{\lim}\frac{1}{p^n}D=0,$$ we get $F(R,R,I)(y)= F(R,R,J)(y)$.\\\\
(2) Let $P_1, \ldots, P_t$ be the minimal primes of $R$. For a finite co-length homogeneous ideal $\mathfrak{a}$, denote the Hilbert-Kunz multiplicity of the triple $(R,R, \mathfrak{a})$ (see \Cref{HK multiplicity}) by $e_{\text{HK}}(R,\mathfrak{a})$. Since all the minimal primes of $R$ have the same dimension, evaluating the equality in \Cref{FP functions are described by generic ranks} at $y=0$ and using \Cref{Hk multiplicity from the Frobenis-Poincare function}, we get
\begin{equation}\label{localization formula for HK multiplicity}
e_{HK}(R,\mathfrak{a})= \sum \limits_{j=1}^{t}e_{HK}(\frac{R}{P_j}, \mathfrak{a}\frac{R}{P_j}) \ .
\end{equation}
Since for each $j$, where $1 \leq j \leq t$, $e_{HK}(\frac{R}{P_j}, I\frac{R}{P_j}) \geq e_{HK}(\frac{R}{P_j}, J\frac{R}{P_j})$ and $e_{HK}(R,I)= F(R,R,I)(0)= F(R,R,J)(0)=e_{HK}(R,J)$, using \eqref{localization formula for HK multiplicity}, we conclude that for each minimal prime $P_j$, $e_{HK}(\frac{R}{P_j}, I\frac{R}{P_j}) = e_{HK}(\frac{R}{P_j}, J\frac{R}{P_j})$. From here we show that the tight closure of $I\frac{R}{P_j}$ and $J\frac{R}{P_j}$ in $\frac{R}{P_j}$ are the same for any $j$; this coupled with Theorem 1.3, (c) of \cite{HunekeTight} establishes that $I^*= J^*$. To this end, fix a minimal prime $P_j$. First note that $\hat{\frac{R}{P_j}}$: the completion of $\frac{R}{P_j}$ at the homogeneous maximal ideal is a domain. To see this, set $I_n \subseteq \frac{R}{P_j}$ to be the ideal generated by forms of degree at least $n$. Then the associated graded ring of $\hat{\frac{R}{P_j}}$ with respect to the filtration $(I_n\hat{\frac{R}{P_j}})_n$ is isomorphic to the domain $\frac{R}{P_j}$- so by Theorem 4.5.8, \cite{BH} $\hat{\frac{R}{P_j}}$ is a domain. Set $m_j$ to be the maximal ideal of $\frac{R}{P_j}$. Since $\hat{\frac{R}{P_j}}$ is a domain, by Theorem 5.5, \cite{HunekeExp} $(I(\frac{R}{P_j})_{m_j})^*= (J(\frac{R}{P_j})_{m_j})^*$. Since both $I$ and $J$ are $m_j$-primary, by Theorem 1.5, \cite{HunekeTight}, we conclude that the tight closures of $I\frac{R}{P_j}$ and $J\frac{R}{P_j}$ in $\frac{R}{P_j}$ are the same.
\end{proof}
Next, we set to show that over a standard graded ring, our Frobenius-Poincar\'e functions are holomorphic Fourier transforms of Hilbert-Kunz density functions introduced in \cite{TriExist}. We first recall a part of a result in \cite{TriExist} that implies the existence of Hilbert-Kunz density functions.
\begin{theorem}\label{HK density function}(see Theorem 1.1 and Theorem 2.19 \cite{TriExist})
Let $k$ be a field of characteristic $p>0$, $R$ be a standard graded $k$-algebra of Krull dimension $d \geq 1$, $I$ be a homogeneous ideal of finite co-length. Given a finitely generated $\N$-graded $R$-module $M$, consider the sequence $(g_n)_n$ of real valued functions defined on the real line where
\begin{equation}\label{formula for density function}
    g_n(x)= (\frac{1}{p^n})^{d-1}\lambda_k((\frac{M}{\bp{I}{n}M})_{\lfloor x p^n \rfloor}).
\end{equation}

Then
\begin{enumerate}
    \item There is a compact subset of the non-negative real line containing the support of $g_n$ for all $n$.
    \item  The sequence $(g_n)$ converges pointwise to a compactly supported function $g$. Furthermore, when $d \geq 2$, the convergence is uniform and $g$ is continuous.
\end{enumerate}
\end{theorem}
The function $g$ in \Cref{HK density function} is called the \textit{Hilbert-Kunz density function} associated to the triple $(M,R,I)$.\\\\
Recall that the \textit{holomorphic Fourier transform} of a compactly supported Lebesgue integrable function $h$ defined on the real line is the holomorphic function $\hat{h}$ given by
$$\hat{h}(y)= \underset{\mathbb{R}}{\int} h(x)e^{-iyx}dx, $$ where the integral is a \textit{Lebesgue integral} (see Chapter 2, \cite{Rudin}).
\begin{proposition}\label{FP function is fourier transform of density function}
The holomorphic Fourier transform of the Hilbert-Kunz density function associated to a triple $(M,R,I)$ as in \Cref{HK density function} is the Frobenius-Poincar\'e function $F(M,R,I,d)$.
\end{proposition}
\begin{proof}
Let $g_n$ and $g$ be as in \Cref{HK density function}. We first establish the claim that there is a constant $C$, such that for any real number $x$ and all $n$, $g_n(x) \leq C$. We can assume that there is  compact subset $[0,N]$ containing the support of $g_n$ for all $n$ (see (1), \Cref{HK density function}). Now given $x$ where $\frac{1}{p^n} \leq x \leq N$,
$$g_n(x) \leq (\frac{1}{p^n})^{d-1}\lambda(M_{\lfloor xp^n \rfloor}) = (\frac{{\lfloor xp^n \rfloor}}{p^n})^{d-1} \frac{\lambda(M_{\lfloor xp^n \rfloor})}{({\lfloor xp^n \rfloor})^{d-1}} \leq N^{d-1}\frac{\lambda(M_{\lfloor xp^n \rfloor})}{({\lfloor xp^n \rfloor})^{d-1}}.$$
Since the function $\frac{\lambda(M_m)}{m^{d-1}}$ is bounded above by a constant (see Proposition 4.4.1 and Exercise 4.4.11 of \cite{BH}), the claim follows.

The bound on $g_n$ allows us to use dominated convergence theorem to the sequence $(g_n)_n$, which implies that the sequence of functions $(\hat{g_n})_n$ converges to $\hat{g}$ pointwise. Now we claim that the sequence $(\hat{g_n})$ in fact converges to  $F(M,R,I,d)$ pointwise; this would imply $\hat{g}= F(M,R,I,d)$.\\
Now for a non-zero complex number $y$,
\begin{equation}\label{Fourier transform of the sequence}\begin{split}
\hat{g_n}(y)&=(\frac{1}{p^n})^{d-1} \int_{0}^{\infty} \lambda_k((\frac{M}{\bp{I}{n}M})_{\lfloor x p^n \rfloor})e^{-iyx}dx\\
&= (\frac{1}{p^n})^{d-1} \sum \limits_{j=0}^{\infty} \int_{\frac{j}{p^n}}^{\frac{j+1}{p^n}}\lambda_k((\frac{M}{\bp{I}{n}M})_{j})e^{-iyx}dx\\
&= (\frac{1}{p^n})^{d-1} \sum \limits_{j=0}^{\infty} \lambda_k((\frac{M}{\bp{I}{n}M})_{j})\left(\frac{e^{-iy(j+1)/p^n}-e^{-iyj/p^n}}{-iy}\right)\\
&= (\frac{1}{p^n})^{d-1} \sum \limits_{j=0}^{\infty} \lambda_k((\frac{M}{\bp{I}{n}M})_{j}) e^{-iyj/p^n}\left(\frac{e^{-iy/p^n}- 1}{-iy}\right)\\
&= (\frac{1}{p^n})^{d} \sum \limits_{j=0}^{\infty} \lambda_k((\frac{M}{\bp{I}{n}M})_{j}) e^{-iyj/p^n}\left(\frac{e^{-iy/p^n}- 1}{-iy/p^n}\right) \ .
\end{split}
\end{equation}
So using the last line of \eqref{Fourier transform of the sequence} and \Cref{the essential limit}, we get that for a non-zero complex number $y$,
$$\hat{g}(y)= \underset{n \to \infty}{\lim}\hat{g_n}(y)= \underset{n \to \infty}{\lim}(\frac{1}{p^n})^{d} \sum \limits_{j=0}^{\infty} \lambda_k((\frac{M}{\bp{I}{n}M})_{j}) e^{-iyj/p^n}= F(M,R,I,d)(y).$$
Note that,
\begin{equation}\label{value at zero}
\hat{g_n}(0)= (\frac{1}{p^n})^d\sum \limits_{j=0}^{\infty} \lambda((\frac{M}{\bp{I}{n}M})_j)= (\frac{1}{p^n})^d\lambda(\frac{M}{\bp{I}{n}M})= F_n(M,d)(0) \ .
\end{equation}
Taking limit as $n$ approaches infinity in \eqref{value at zero} gives $\hat{g}(0)= F(M,R,I,d)(0).$
\end{proof}
\begin{remark}\label{comparison between FP function and HK density function}
\begin{enumerate}
    \item Since a compactly supported continuous function can be recovered from its holomorphic Fourier transform (see Theorem 1.7.3, \cite{Hormander}), the existence of Frobenius-Poincar\'e functions gives an alternate proof of the existence of Hilbert-Kunz density functions in dimension $d \geq 2$.
    \item One way to incorporate zero dimensional ambient rings into the theory of Hilbert-Kunz density functions could be to realize the functions $g_n$ in \eqref{formula for density function} and the resulting Hilbert-Kunz density function as \textit{compactly supported distributions} (see Definition 1.3.2, \cite{Hormander}). Here by a \textit{distribution}, we mean a $\mathbb{C}$-linear map from the space of complex valued smooth functions on $\mathbb{R}$ to $\mathbb{C}$. In our case, the distribution defined by each $g_n$ sends the function $f$ to $\underset{\mathbb{R}}{\int}f(x)g_n(x) dx$. When the ambient ring has dimension at least one, the sequence of distributions defined $(g_n)_n$ converges to the distribution defined by the corresponding Hilbert-Kunz density function; see the Remark on page 7 of \cite{Hormander} for a precise meaning of convergence of distributions. Now suppose that $R$ has dimension zero and $M$ is a finitely generated $\Z$-graded $R$-module; let $(g_n)_n$ be the corresponding sequence of functions given by \eqref{formula for density function} with $d=0$. Direct calculation shows that for a complex valued smooth function $f$, the sequence of numbers $\underset{\mathbb{R}}{\int}f(x)g_n(x)dx$ converges to $\lambda_k(M)f(0)$. This means that the sequence of distributions defined by $(g_n)_n$ converges to the distribution $\lambda_k(M)\delta_0$- where $\delta_0$ is the distribution such that $\delta_0(f)= f(0)$. So it is reasonable to define the Hilbert-Kunz density function $g(M,R,I)$ to be the \textit{distribution} $\lambda_k(M)\delta_0$. In fact, incorporating the language of Fourier transform of distributions (see section 1.7, \cite{Hormander}), it follows that the Fourier transform of the Hilbert-Kunz density function (or distribution) is our Frobenius-Poincar\'e function irrespective of the dimension of the ambient ring. Going in the reverse direction, Hilbert-Kunz density function of a triple can be defined to be the unique compactly supported distribution whose Fourier transform is the corresponding Frobenius-Poincar\'e function. 
\end{enumerate}

\end{remark}

\section{Descriptions using Homological Information} \label{description using homological information}
In this section, we give alternate descriptions of Frobenius-Poincar\'e functions of $(R,R,I)$ in terms of the sequence of graded Betti numbers of $R/\bp{I}{n}$. Moreover when $R$ is Cohen-Macaulay, the Frobenius-Poincar\'e functions are described using the Koszul homologies of $\frac{R}{\bp{I}{n}}$ with respect to a homogeneous system of parameters of R. Some background material for this section on Hilbert-Samuel multiplicity, Hilbert series and graded Betti numbers is reviewed in \Cref{Homological} and \Cref{Hilbert series and Hilbert-Samuel multiplicity}.
\begin{theorem}\label{description in terms of Betti number}Let $S$ be a graded complete intersection over $k$ of Krull dimension $d$ and Hilbert-Samuel multiplicity $e_S$ (see \Cref{HS multiplicity}). Let $S \rightarrow R$ be a module finite $k$-algebra map to a finitely generated $\N$-graded $k$-algebra. Let $I \subseteq R$  be  a homogeneous ideal of finite co-length and $M$ be a finitely generated $\Z$-graded $R$-module.
Set
\begin{equation}\label{the sum defining twisted sum of Betti numbers}
\mathbb B^S(j, n) = \sum_{\alpha=0}^{\infty}
 (-1)^{\alpha} \,
 \lambda((\text{Tor}_{\alpha}^S(k, M/\bp{I}{n}M)_j). 
\end{equation}
 
Then
\begin{enumerate}
    \item ${\underset{n \to \infty}{\lim}(p^n)^{d-\text{dim}\, M}\sum \limits_{j=0}^{\infty} \, \mathbb{B}^S(j,n)\,e^{-iyj/p^n}}$ admits an analytic extension to the complex plane.
    \item The Frobenius-Poincar\'e function $F(M,R,I)(y)$ is the same as the analytic extension of the function \\ $\frac {e_S}{(iy)^{d}}\,\, 
{\underset{n \to \infty}{\lim}(p^n)^{d-\text{dim}\, M}\sum \limits_{j=0}^{\infty} \, \mathbb{B}^S(j,n)\,e^{-iyj/p^n}}$ to the complex plane.
\end{enumerate}
\end{theorem}

Note that for fixed integers $j,n$ the sum in \eqref{the sum defining twisted sum of Betti numbers} is finite- see \Cref{alternating sum of twists}. We record some remarks and consequences related to \Cref{description in terms of Betti number} before proving the result. 
\begin{corollary}\label{the complete intersection case}
For a graded complete intersection $R$ over $k$  and a homogeneous ideal $I \subseteq R$ of finite colength, the function $\underset{n \to \infty}{\lim}\sum \limits_{j=0}^{\infty}\mathbb{B}^R(j,n)e^{-iyj/p^n}$ extends to the entire function $\frac{1}{e_R}(iy)^{\text{dim}(R)}F(R,R,I)(y)$.
\end{corollary}
\begin{remark}
When applied to the triple $(R,R,m)$, where $m$ is the homogeneous maximal ideal of a graded complete intersection $R$ over $k$, \Cref{description in terms of Betti number} applied to the case $S=R=M$, gives a way to compare the Hilbert-Kunz multiplicity of $(R,m)$ to the Hilbert-Samuel multiplicity $e_R$.
\end{remark}
\begin{remark}\label{Betti numbers with respect to Noether normalization}
One way to apply \Cref{description in terms of Betti number} to describe the Frobenius-Poincar\'e function of a graded triple $(M,R,I)$ is to take $S$ to be a subring of $R$ generated by a homogeneous system of parameters. Since such an $S$ is regular, for any integer $n$, the sum defining $\mathbb{B}^S(j,n)$ in \eqref{the sum defining twisted sum of Betti numbers} is finite and  the function $\sum \limits_{j=0}^{\infty} \, \mathbb{B}^S(j,n)\,e^{-iyj/p^n}$ appearing in \Cref{description in terms of Betti number} is a polynomial in $e^{-iy/p^n}$.
\end{remark}
\begin{remark}\label{differentiability of density function}In \cite[page 7]{TriQuadric} V. Trivedi asks whether the Hilbert-Kunz density function (see \Cref{HK density function}) of a $d (\geq 2)$ dimensional standard graded pair $(R,I)$ is always $(d-2)$ times differentiable and the $(d-2)$-th order derivative is continuous. We use \Cref{description in terms of Betti number} to reformulate Trivedi's question and produce the candidate function for the $(d-2)$-th order derivative. Denote the restriction of $(iy)^{d-2}F(R,R,I)(y)$ to the real line by $h$. The Fourier transform of the \textit{temperate distribution} (see Definition 1.7.2, 1.7.3, \cite{Hormander}) defined by $h$ determines the $(d-2)$-th order \textit{derivative of the distribution} (see Definition 1.4.1, \cite{Hormander}) defined by the Hilbert-Kunz density function of $(R,I)$- see \Cref{comparison between FP function and HK density function}, Theorem 1.7.3, \cite{Hormander}. In fact using tools from analysis, one can show that Trivedi's question has an affirmative answer if the integral $\underset{\mathbb{R}}{\int}|h(x)|dx$ is finite\footnote{The argument is recorded in \cite[Theorem 8.3.8]{AlapanThesis}.}. When $\underset{\mathbb{R}}{\int}|h(x)|dx$ is finite, the Fourier transform of $h$ is in fact given by the actual function $\hat{h}(y)= \underset{\mathbb{R}}{\int}h(x)e^{-iyx}dx$ for $y \in \mathbb{R}$ and the $(d-2)$-th order derivative is the function $\frac{1}{2 \pi}\hat{h}(-y)$. Fix a subring $S \subseteq R $ generated by a homogeneous system of parameters of $R$. Since by \Cref{description in terms of Betti number} applied to the case $M=R$ (also see \Cref{Betti numbers with respect to Noether normalization}) $h(y)= e_S \underset{n \to \infty}{\lim} \frac{\sum \limits_{j=0}^{\infty}\mathbb{B}^S(j,n)e^{-iyj/p^n}}{(iy)^2}$, it is natural to ask
\begin{question}\label{boundedness of holomorphic function defined by Betti numbers}
\begin{enumerate}
    \item Is the function $ h(y)= \frac{\underset{n \to \infty}{\lim} \sum \limits_{j=0}^{\infty}\mathbb{B}^S(j,n)e^{-iyj/p^n}}{(iy)^2}$ integrable on $\mathbb{R}$?
    \item Is function $\underset{n \to \infty}{\lim} \sum \limits_{j=0}^{\infty}\mathbb{B}^S(j,n)e^{-iyj/p^n}$ restricted to the real line bounded?
\end{enumerate}
\end{question}
Note that an affirmative answer to part (2) implies an affirmative answer to part (1) of the \Cref{boundedness of holomorphic function defined by Betti numbers}.
\end{remark}
\medskip
We use a consequence of a result from \cite{AB}-where it is cited as a folklore- in the proof of \Cref{description in terms of Betti number} below.
\begin{proposition}\label{AvrBuch}(see Lemma 7.ii, \cite{AB})
Let $R$ be a finitely generated $\N$-graded $k$-algebra and $M,N$ be two finitely generated $\Z$-graded $R$-modules. Denote the formal Laurent series $\underset{i \in \N}{\sum} (-1)^i H_{\text{Tor}_i^R(M,N)}(t)$ by $\chi^R(M,N)(t)$. Then
$$\chi^R(M,N)(t)= \frac{H_M(t)H_N(t)}{H_R(t)},$$
where for a finitely generated $\Z$-graded $R$-module $N'$, $H_{N'}(t)$ is the Hilbert series of $N'$.
\end{proposition}
\begin{proof}[\textbf{Proof of Theorem \ref{description in terms of Betti number}:}] Let $\mathfrak{H}$ be the set of complex numbers with a negative imaginary part. We shall prove that on the connected open subset $\mathfrak{H}$ of the complex plane $$\frac {e_S}{(iy)^{d}}\,\, 
{\underset{n \to \infty}{\lim}(p^n)^{d-\text{dim}\, M}\sum \limits_{j=0}^{\infty} \, \mathbb{B}^S(j,n)\,e^{-iyj/p^n}}$$ defines a holomorphic function and is the same as the restriction of $F(R,R,I)$ to $\mathfrak{H}$. Since $F(R,R,I)$ is an entire function, the analytic continuity in assertion 1 and the desired equality in assertion 2 follows.

Given an integer $n$, $\chi^S(\frac{M}{\bp{I}{n}M}, k)(t)= \sum \limits_{j= - \infty}^{\infty}\mathbb{B}^S(j,n)t^j$. So using \Cref{AvrBuch} we get
\begin{equation}\label{AB in the complete intersection case}
H_{\frac{M}{\bp{I}{n}M}}(t)= H_S(t)(\sum \limits_{j= - \infty}^{\infty}\mathbb{B}^S(j,n)t^j).
\end{equation}
Now for any $y \in \mathfrak{H}$, $|e^{-iy/p^n}|<1$; so by \Cref{radius of conv of twisted Poincare series}, the series $\underset{j \in \Z}{\sum} \mathbb{B}^S(j,n)(e^{-iy/p^n})^j$ converges absolutely. For $y \in \mathfrak{H}$, plugging in $t= e^{-iy/p^n}$ in \eqref{AB in the complete intersection case}, we get
\begin{equation}\label{holomorphic version of description in terms of bett numbers}
\begin{split}
(\frac{1}{p^{n}})^{\text{dim}(M)}H_{\frac{M}{\bp{I}{n}M}}(e^{-iy/p^n}) &= (\frac{1}{p^{n}})^{d}H_S(e^{-iy/p^n})(p^n)^{d-\text{dim}(M)}(\sum \limits_{j= - \infty}^{\infty}\mathbb{B}^S(j,n)e^{-iyj/p^n})\\
&= \frac{(1-e^{-iy/p^n})^d}{(p^n(1-e^{-iy/p^n}))^d}H_S(e^{-iy/p^n})(p^n)^{d-\text{dim}(M)}(\sum \limits_{j= - \infty}^{\infty}\mathbb{B}^S(j,n)e^{-iyj/p^n}) \ .
\end{split}
\end{equation}
For a fixed $y \in \mathfrak{H}$, as $n$ approaches infinity, $(1-e^{-iy/p^n})^dH_S(e^{-iy/p^n})$ approaches $e_S$ (see \Cref{Samuel multiplicity in terms of Hilbert-Poincare series}) and $(p^n(1-e^{-iy/p^n}))^d$ approaches $(iy)^d$(see \Cref{the essential limit}).
Now taking limit as $n$ approaches infinity in \eqref{holomorphic version of description in terms of bett numbers} gives the following equality on $\mathfrak{H}$:
$$F(R,R,I)(y)= \frac {e_S}{(iy)^{d}}\,\, 
{\underset{n \to \infty}{\lim}(p^n)^{d-\text{dim}\, M}\sum \limits_{j=0}^{\infty} \, \mathbb{B}^S(j,n)\,e^{-iyj/p^n}}.$$
Since the left hand side of the last equation is holomorphic on $\mathfrak{H}$, so is the right hand side; this finishes the proof.
\end{proof}
\vspace{.5cm}

\begin{remark}
Take $S=R=M$ in \Cref{description in terms of Betti number} and let $\mathfrak{H}$ be the same as in the proof of \Cref{description in terms of Betti number}. Although the analyticity of $\sum \limits_{j=0}^{\infty}\mathbb{B}(j,n)e^{-iyj/p^n}$ on $\mathfrak{H}$, for each $n$,
follows from \Cref{radius of conv of twisted Poincare series}, the existence of the analytic extension of their limit crucially depends on \Cref{description in terms of Betti number} and that Frobenius-Poincar\'e functions are entire.
\end{remark}
\medskip
When the $R$-module $R/I$ has finite projective dimension, the line of argument in \Cref{description in terms of Betti number} (also see \cite{TriDomain}) allows to describe $F(R,R,I)$ in terms of the graded Betti numbers of $R/I$.
\begin{proposition}
\label{finite projective dimesnion case}
Let $I$ be a homogeneous ideal of the $d$ dimensional ring $R$, such that the projective dimension of the $R$-module $R/I$ is finite. Set
$$b_{\alpha,j}= \lambda(\text{Tor}_{\alpha}^R(k, R/I)_j), \hspace{.25cm} \mathbb{B}(j)= \sum \limits_{\alpha=0}^{\infty}(-1)^{\alpha}b_{\alpha,j}, \hspace{.25cm} e_R = \text{Hilbert-Samuel multiplicity of}\, R.$$
 Let $b$ be the smallest integer such that $\mathbb{B}(j)=0$ for all $j >b$. Then for a non-zero complex number $y$, we have:
$$F(R,R,I)(y)= e_R \ \frac{\sum \limits_{j=0}^{b}\mathbb{B}(j)e^{-iyj}}{(iy)^d}.$$
\end{proposition}
\begin{proof}
Take a minimal graded free resolution of the $R$-module $R/I$: $$
\begin{tikzcd}
0 \arrow[r] & \underset{j \in \N}{\oplus}R(-j)^{\oplus b_{d,j}} \arrow[r] & {} & \ldots \arrow[r] & \underset{j \in \N}{\oplus}R(-j)^{\oplus b_{1,j}} \arrow[rr] &  & \underset{j \in \N}{\oplus}R(-j)^{\oplus b_{0,j}} \arrow[r] & R/I \arrow[r] & 0 .
\end{tikzcd}$$
Then we get a minimal graded free resolution of $R/\bp{I}{n}$ by applying $n$-th iteration of the Frobenius functor to the chosen minimal graded resolution of $R/I$ (see Theorem 1.13 of \cite{PS}),
$$
\begin{tikzcd}
0 \arrow[r] & \underset{j \in \N}{\oplus}R(-p^nj)^{\oplus b_{d,j}} \arrow[r] & {} & \ldots \arrow[r] & \underset{j \in \N}{\oplus}R(-p^nj)^{\oplus b_{0,j}} \arrow[r] & {R/\bp{I}{n}} \arrow[r] & 0 .
\end{tikzcd}$$
So using the notation set in \eqref{the sum defining twisted sum of Betti numbers} in the case $S=R=M$ and the ideal $I$, we have that for any positive integer $n$, $\mathbb{B}^R(jp^n,n)= \mathbb{B}(j)$ and $\mathbb{B}(m,n)=0$ if $p^n$ does not divide $m$.
So for all $n \in \N$, $\chi^R(\frac{R}{\bp{I}{n}},k)(t)=\underset{j \in \N}{\sum}\mathbb{B}^R(j,n)t^j= \sum \limits_{j=0}^{b}\mathbb{B}(j)t^{jp^n}$ is a polynomial in $t$. So for any $n\in \N$, using \Cref{AvrBuch} for $M= R/ \bp{I}{n}$, $N= k$ we have for all $y \in \mathbb{C}$,
\begin{equation}\label{equation finite pd case}
(\frac{1}{p^n})^d H_{\frac{R}{\bp{I}{n}}}(e^{-iy/p^n})= \frac{(1-e^{-iy/p^n})^dH_R(e^{-iy/p^n})}{(p^n(1-e^{-iy/p^n}))^d}(\sum \limits_{j=0}^{b}\mathbb{B}(j)e^{-iyj}).    
\end{equation}
Now taking limit as $n$ approaches infinity in \eqref{equation finite pd case} and using \Cref{Samuel multiplicity in terms of Hilbert-Poincare series} and \Cref{the essential limit}, we get
$$F(R,R,I)(y)= e_R \frac{\sum \limits_{j=0}^b \mathbb{B}(j)e^{-iyj}}{(iy)^d}.$$
\end{proof}
\begin{example}\label{fine structure}
Let $R= k[X,Y]$ be the standard graded polynomial ring in two variables, $I= (f,g)R$ where $f$ and $g$ have degree $d_1$ and $d_2$ respectively. Using \Cref{finite projective dimesnion case}, we can compute $F(R,I)(y)$. A minimal free resolution of $R/I$ is given by the Koszul complex of $(f,g)$:
$$0 \rightarrow R(-d_1-d_2) \rightarrow R(-d_1) \oplus R(-d_2) \rightarrow R \rightarrow 0 .$$
Hence we get,
\begin{equation}
\begin{split}
F(R,I)(y) &= \frac{\mathbb{B}(0)+ \mathbb{B}(d_1)e^{-iyd_1}+      \mathbb{B}(d_2)e^{-iyd_2}+ \mathbb{B}(d_1+d_2)e^{-iy(d_1+d_2)}}{(iy)^2} \\
          &= \frac{1-e^{-iyd_1}- e^{-iyd_2}+ e^{-iy(d_1+d_2)}}{(iy)^2}.
\end{split}
\end{equation}
The Hilbert-Kunz multiplicity $e_{HK}(R,I)= d_1d_2$ (see for example Theorem 11.2.10 of \cite{HunekeSwanson}). Using this observation, we can construct finite co-length ideals $I$ and $J$ in $R$ such that, $e_{HK}(R,I)= e_{HK}(R,J)$ but $F(R,R,I)$ and $F(R,R,J)$ are different.
\end{example}

\noindent In the next result, we show that the Frobenius-Poincar\'e function of a Cohen-Macaulay ring can be described in terms of the sequence of Koszul homologies of $\frac{R}{\bp{I}{n}}$ with respect to a homogeneous system of parameters. 
\begin{theorem}\label{cohen-macaulay case}
Let $R$ be a Cohen-Macaulay $\N$-graded ring of dimension $d$, $I$ be a homogeneous ideal of finite co-length of $R$. Let $x_1, x_2, \ldots, x_d$ be a homogeneous system of parameters of $R$ of degree $\delta_1, \ldots, \delta_d$ respectively. Then
$$F(R,I)(y)= \frac{1}{\delta_1 \delta_2 \ldots \delta_d(iy)^d} \lim_{n \to \infty} \chi^R(\frac{R}{(x_1, \ldots,x_d)R}, \frac{R}{\bp{I}{n}R}) (e^{-iy/p^n}),$$
where $\chi^R(\frac{R}{(x_1, \ldots,x_d)R}, \frac{R}{\bp{I}{n}R}) (t)$ has the same meaning as in \Cref{AvrBuch}.
\end{theorem}
\noindent\begin{remark}
The Laurent series $\chi^R(M,N)(t)$ for a pair of graded modules as defined in \Cref{AvrBuch} has been used before to define multiplicity or intersection multiplicity in different contexts; see for example \cite{Serre} Chapter IV, A, Theorem 1 and \cite{Erman}. The assertion in \Cref{cohen-macaulay case} should be thought of as an analogue of the these results since here the Frobenius-Poincar\'e function and hence the Hilbert-Kunz multiplicity is expressed in terms of the limit of power series $\chi^R(\frac{R}{\bp{I}{n}},\frac{R}{(\underline{x})R})(t)$. 
\end{remark}
\begin{proof}[\textbf{Proof of \Cref{cohen-macaulay case}:}]
Using \Cref{AvrBuch} we have,
\begin{equation}\label{applying avr in the Cohen-Macaulay case}
    H_{\frac{R}{\bp{I}{n}}}(e^{-iy/p^n})H_{\frac{R}{(x_1, \ldots, x_d)}}(e^{-iy/p^n})= H_R(e^{-iy/p^n})\chi^R(\frac{R}{(x_1, \ldots,x_d)R}, \frac{R}{\bp{I}{n}R}) (e^{-iy/p^n}) \ .
\end{equation}
Since $R$ is Cohen-Macaulay, $x_1, \ldots, x_d$ is a regular sequence. Inducing on $d$, one can show that,
\begin{equation}\label{hilbert series modulo a nzd}
H_{\frac{R}{(x_1, \ldots, x_d)}}(t)= (1-t^{\delta_1})(1-t^{\delta_2})\ldots (1-t^{\delta_d})H_R(t) \ .
\end{equation}
Using \eqref{hilbert series modulo a nzd} in \eqref{applying avr in the Cohen-Macaulay case} we get
$$(\frac{1}{p^n})^dH_{R/\bp{I}{n}}(e^{-iy/p^n})= \frac{\chi^R(\frac{R}{(x_1, \ldots,x_d)R}, \frac{R}{\bp{I}{n}R}) (e^{-iy/p^n})}{(p^n)^d (1-e^{-iy\delta_1/p^n})\ldots (1-e^{-iy\delta_d/p^n})}.$$
The desired assertion now follows from taking limit as $n$ approaching infinity in the last equation and using \Cref{the essential limit}.
\end{proof}
In each of \Cref{description in terms of Betti number}, \Cref{finite projective dimesnion case}, \Cref{cohen-macaulay case}, the Frobenius-Poincar\'e function is described as a quotient: the denominator is a power of $iy$ and the numerator is a limit of a sequence of power series or polynomials in $e^{-iy/p^n}$. In particular, in \Cref{cohen-macaulay case}, the maximum value of $j/p^n$ that appears in $e^{-iyj/p^n}$ in the sequence of functions is bounded above by a constant independent of $n$- see \Cref{support}. So we ask
\begin{question} \label{the conjectural structure} 

Let $R$ be a Cohen-Macaulay $\N$-graded ring of dimension $d$, $I$ is a homogeneous ideal of finite co-length. Does there exist a real number $r$ and a polynomial $Q \in \mathbb{R}[X]$ such that
$$F(R,R,I)(y)= \frac{Q(e^{-iry})}{(iy)^d} ?$$
\end{question}
\section{Frobenius-Poincar\'e functions in dimension two}\label{in dimension two}
We compute the Frobenius-Poincar\'e function of two dimensional graded rings following the work of \cite{Brrationality} and \cite{TriCurves}. In this section, $R$ stands for a normal, two dimensional, standard graded domain. We assume that $R_0=k$ is an algebraically closed field of prime characteristic $p$. The smooth embedded curve $\text{Proj}(R)$ is denoted by $C$, $\delta_R$ stands for the Hilbert-Samuel multiplicity of $R$; alternatively $\delta_R$ is the degree of the line bundle $\mathcal{O}_C(1)$. The genus of $C$ is denoted by $g$. For a sheaf of $\mathcal{O}_C$-modules $\F$, $\F(j)$ stands for the sheaf $\F \otimes \mathcal{O}_C(j)$. The absolute Frobenius endomorphism of $C$ is denoted by $f$. Some background materials for this section are reviewed in \Cref{vector bundles on curve}.
\begin{theorem}\label{statement in the curve case}
With notation as in the paragraph above, let $I$ be an ideal of finite colength in $R$ generated by degree one elements $h_1, h_2, \ldots, h_r$. Consider the short exact sequence of vector bundles on $C= Proj(R)$
\begin{equation}\label{short exact sequence on curve}
0 \longrightarrow \mathcal {S} \longrightarrow \overset{r}{\bigoplus} \mathcal {O}_C \xrightarrow{(h_1, \ldots, h_r)} \mathcal{O}_C(1) \longrightarrow 0 \ .
\end{equation}
Choose $n_0$ such that the Harder-Narasimhan filtration on $f^{n_0*}(\mathcal{S})$ given by
\begin{equation}
 0=\E_0 \subset \E_1 \subset \ldots \subset \E_t \subset \E_{t+1}= f^{n_0*}\mathcal{S} \
\end{equation}
is strong \footnote{that is the pull back of the Harder-Narasimhan filtration on $(f^{n_0})^*\mathcal{S}$ via $f^{n-n_0}$ gives the the Harder-Narasimhan filtration on $(f^n)^*\mathcal{S}$- see \Cref{HN filtration}, \Cref{Langer}.}.
For any $1\leq s \leq t+1$, set $\mu_s$ to be the normalized slope $\frac{\mu(\E_s/\E_{s-1})}{p^{n_0}}$ of the factor   $\E_s/\E_{s-1}$  and set $r_s$ to be its rank $\text{rk}(\E_s/\E_{s-1})$ (see \Cref{subbundle}).
Then 
\begin{equation}\label{curve}
    F(R, I)(y)= \delta_R\frac{1-(1+ \text{rk}(\mathcal{S}))e^{-iy}+ \sum \limits_{j=1}^{t+1}r_je^{-iy(1-\frac{\mu_j}{\delta_R})}}{(iy)^2} \ .
\end{equation}
\end{theorem}
\begin{remark}
The two relations $\text{rk}(\mathcal{S})= \sum \limits_{j=1}^{t+1}r_j$ and $\sum \limits_{j=1}^{t+1}\mu_jr_j=- \delta_R$ imply that the numerator of the right hand side of \eqref{curve} has a zero of order two at the origin. So the right hand side of \eqref{curve} is holomorphic at the origin. Conversely, the holomorphicity of the Frobenius-Poincar\'e function and the equality \eqref{curve} for non-zero complex numbers reveal the two relations. 
\end{remark}

The key steps in the proof of \Cref{statement in the curve case} are \Cref{length in terms of cohomologies lemma} and \Cref{cohomologies in terms of HN filtration}. The standard reference for results on sheaf cohomology used here is \cite{Hart}.

\begin{lemma}\label{length in terms of cohomologies lemma}
For $j > 2g-2$ and for all $n$, we have
\begin{equation}\label{length interms of cohomologies}
  \lambda((\frac{R}{\bp{I}{n}})_{p^n+j})= h^1(C, (f^{n*}\mathcal {S})(j))  \ .
\end{equation}
\end{lemma}
\begin{proof}
Given natural numbers $n$ and $j$, first pulling back \eqref{short exact sequence on curve} via $f^n$ and then tensoring with $\mathcal{O}_C(j)$, we get a short exact sequence:
\begin{equation}\label{short exact}
    0 \longrightarrow (f^{n*}\mathcal {S})(j) \longrightarrow \overset{r}{\bigoplus} \mathcal {O}_C(j) \xrightarrow{(h_1^{p^n}, \ldots, h_r^{p^n})} \mathcal{O}_C(p^n+j) \longrightarrow 0 \ .
\end{equation}
Note that, since $R$ is normal, for each $j$, the canonical inclusion $R_j \subset h^0(\mathcal {O}_C(j))$ is an isomorphism - see Exercise 5.14 of \cite{Hart}. Also for $j > 2g-2$, $h^1(\mathcal{O}_C(j))=0$ (see Example 1.3.4 of \cite{Hart}). So the long exact sequence of sheaf cohomologies corresponding to \eqref{short exact} gives \eqref{length interms of cohomologies}.  
\end{proof}
\begin{lemma}\label{cohomologies in terms of HN filtration}
Fix $s$ such that $1 \leq s \leq t$. For  all large $n$, if an index $j$ satisfies
\begin{equation}\label{range of twist}
    -p^n\frac{{\mu}_{s}}{\delta_R}+ \frac{2g-2}{\delta_R} < j < -p^n \frac{{\mu}_{s+1}}{\delta_R}\ ,
\end{equation}
then
\begin{equation}
    h^1(f^{n*}(\mathcal{S})(j))= -p^n(\sum \limits_{b=s}^t {\mu}_{b+1}r_{b+1})+ (\sum \limits_{b=s}^{t}r_{b+1})(g-1-j
    \delta_R) \,  .
\end{equation}
And for $j> -p^n \frac{\mu_{t+1}}{\delta_R}$,
$$ h^1(f^{n*}(\mathcal{S})(j))=0.$$
\end{lemma}
\vspace{.5cm}
\noindent The only reason for choosing large $n$ is to ensure that for all $s$
$$ -p^n\frac{{\mu}_{s}}{\delta_R}+ \frac{2g-2}{\delta_R} < -p^n \frac{{\mu}_{s+1}}{\delta_R}.$$
\begin{proof}
We introduce some notation below which are used in this section. \\\\
Set $K_j = \text{ker}((f^{n_0*}\mathcal S)^ \vee \rightarrow \E_{t+1-j}^ \vee)$. Then by \Cref{filtration on the dual}, there is a HN filtration-
\begin{equation}
0=K_0 \subset K_1 \subset \ldots K_t \subset K_{t+1}=  (f^{n_0*}\mathcal{S})^ \vee  
\end{equation} \ .
\vspace{.05cm}
\begin{claim}\label{length in terms of dual lemma} Denote the sheaf of differentials of $C$ by $\omega_C$.
For $j$ as in \eqref{range of twist}, we have
\begin{equation}\label{length in terms of the dual bundle}
(1) \  h^1(f^{n*}(\mathcal{S})(j))= h^0(f^{n-n_0*}(K_{t+1-s}) \otimes \omega_C(-j)) \hspace{.5cm} \text{and} \hspace{.5cm} (2) \ h^1(f^{n-n_0*}(K_{t+1-s}) \otimes \omega_C(-j)) =0 \ .
\end{equation}
\end{claim}
\vspace{.5cm}
\noindent We defer proving the above claim until deriving \Cref{cohomologies in terms of HN filtration} from it.\\\\
Combining \eqref{length in terms of the dual bundle} and the Riemann-Roch theorem on curves (see Theorem 2.6.9, \cite{Potier}), we get that for the range of values as in \eqref{range of twist},
\begin{equation}\label{using Riemann-Roch}
\begin{split}
      h^1(f^{n*}(S)(j)) 
      & = h^0(f^{n-n_0*}(K_{t+1-s})\otimes \omega_C(-j))- h^1(f^{n-n_0*}(K_{t+1-s})\otimes \omega_C(-j))\\
      & = \text{deg}(f^{n-n_0*}(K_{t+1-s})\otimes \omega_C(-j))+ (1-g)\text{rk}(f^{n-n_0*}(K_{t+1-s})\otimes \omega_C(-j))\\
      & = \text{deg}(f^{n-n_0*}K_{t+1-s})+ \text{rk}(K_{t+1-s})\cdot\text{deg}(\omega_C(-j))+ (1-g)\text{rk}(K_{t+1-s}) \ (\text{using}, (1), \Cref{properties of HN filtration}) \ .
\end{split}
\end{equation}
Since $K_{t+1-s}$ is the kernel of a surjection $(f^{n_0*}S)^ \vee \rightarrow \E_s^ \vee$, we have,
\begin{equation}\label{deg of dual}
\begin{split}
   \text{deg}(f^{n-n_0*}K_{t+1-s})
  & = p^{n-n_0}\text{deg}(K_{t+1-s})= p^{n-n_0}[\text{deg}(f^{n_0*}S)^ \vee)- \text{deg}(\E_s^ \vee)]\\ 
  & = p^{n-n_0}[- \text{deg}(f^{n_0*}S)+ \text{deg}(\E_s)]\\
  & = -p^{n-n_0} \sum \limits_{b=s}^{t}\text{deg}(\E_{b+1}/\E_b) = -p^n\sum \limits_{b=s}^{t}{\mu}_{b+1} r_{b+1}.
\end{split}
\end{equation}
Similarly one can compute the rank of $K_{t+1-s}$.
\begin{equation}\label{rank of dual}
  \begin{split}
      \text{rk}(K_{t+1-s})
      & = \text{rk}((f^{n_0*}S)^ \vee)- \text{rk}(\E_s^ \vee)= \text{rk}(f^{n_0*}S)- \text{rk}(\E_s)\\
      & = \sum \limits_{b=s}^{t} \text{rk}(\E_{b+1}/ \E_b) = \sum \limits_{b=s}^{t} r_{b+1}.
  \end{split}  
\end{equation}
Now the desired conclusion follows from combining \eqref{using Riemann-Roch}, \eqref{deg of dual}, \eqref{rank of dual} and noting that $\text{deg}(\omega_C(-j))= 2g-2-j
\delta_R$.\\\\
\noindent \begin{proof}[\textbf{Proof of Claim:}] By Serre Duality (see Corollary 7.7, Chapter III, \cite{Hart}), $h^1(f^{n*}(\mathcal{S})(j))= h^0((f^{n*}\mathcal{S})^ \vee \otimes \omega_C (-j))$. We prove (1) by showing that if the left most inequality in \eqref{range of twist} holds, the cokernel of the inclusion $$H^0(f^{n-n_0*}(K_{t+1-s})\otimes \omega_C(-j)) \subseteq H^0((f^{n*}\mathcal{S})^ \vee \otimes \omega_C (-j))$$
is zero. For this, first note that by \Cref{properties of HN filtration},(3), there is a HN filtration -
\begin{equation}\label{HN filtration on the quotient}
   0 \subset \frac{f^{n-n_0*}(K_{t+2-s})\otimes \omega_C(-j)}{f^{n-n_0*}(K_{t+1-s})\otimes \omega_C(-j)} \subset \ldots \subset \frac{f^{n-n_0*}(K_{t})\otimes \omega_C(-j)}{f^{n-n_0*}(K_{t+1-s})\otimes \omega_C(-j)}\subset \frac{(f^{n*}\mathcal{S})^ \vee \otimes \omega_C (-j)}{f^{n-n_0*}(K_{t+1-s})\otimes \omega_C(-j)}.
\end{equation}
The slope of the first non-zero term in the HN filtration in \eqref{HN filtration on the quotient} is $-\mu_s p^n+ 2g-2- j\delta_R$, which is negative by \eqref{range of twist}. The desired conclusion now follows from \Cref{minimum slope negative implies no global sections}.

Now we show that if $j$ satisfies the right most inequality in \eqref{range of twist}, then assertion (2) in the claim holds. By Serre duality $h^1(f^{n-n_0*}(K_{t+1-s})\otimes \omega_C(-j))= h^0((\frac{f^{n*}S}{f^{n-n_0*}\E_s})(j))$. By \Cref{properties of HN filtration}, (3), the HN filtration on
$(\frac{f^{n*}S}{f^{n-n_0*}\E_s})(j)$ is as given below-
\begin{equation}
    0 \subset \frac{f^{n-n_0*}\E_{s+1}}{f^{n-n_0*}\E_s}(j) \subset \ldots \subset \frac{f^{n-n_0*}\E_{t}}{f^{n-n_0*}\E_{s}}(j) \subset \frac{f^{n*}S}{f^{n-n_0*}\E_s}(j).
\end{equation}
Since the slope of the first non-zero term in the above filtration is $p^n\mu_{s+1}+ j \delta_R$, which negative by \eqref{range of twist}, using \Cref{minimum slope negative implies no global sections}, we get the desired conclusion.
\medskip
\end{proof}
\end{proof}
\begin{proof} [\textbf{Proof of \Cref{statement in the curve case}:}] We shall show that \eqref{curve} holds for all non-zero $y$. Then by the principle of analytic continuation (see page 127, \cite{Ahlfors}), we get \eqref{curve} at all points.\\\\
Fix an open subset $U$ of the complex plane whose closure is compact and the closure does not contain the origin. We fix some notations below which we use in the ongoing proof. For $1 \leq s \leq t+1$, set,
$$
l_s(n)= \lfloor -\mu_s.\frac{p^n}{\delta_R} \rfloor \,\,\, \text{and} \,\,\,  u_s(n)= \lceil-\mu_s.\frac{p^n}{\delta_R}+2g-2 \rceil +1.
$$ 
Note that 
\begin{equation}\label{two limits}
    \underset{n \to \infty}{\lim} \frac{l_s(n)}{p^n}= \frac{u_s(n)}{p^n}= \frac{-\mu_s}{\delta_R} \ .
\end{equation}
There is a sequence of functions $(g_n)_n$ such that for $y \in U$, we have
\begin{equation} \label{splitting into terms}
\begin{split}
    \sum \limits_{j=0}^{\infty} \lambda((\frac{R}{\bp{I}{n}})_j)e^{-iyj/p^n}
    &=\underset{j < p^n}{\sum} \lambda((\frac{R}{\bp{I}{n}})_j)e^{-iyj/p^n}+  \sum \limits_{j=2g-1}^{l_1(n)-1} \lambda((\frac{R}{\bp{I}{n}})_{j+p^n})e^{-iy(j+p^n)/p^n}\\
    & + \sum \limits_{s=1}^{t}  \sum \limits_{j= u_s(n)}^{l_{s+1}(n)-1} \lambda((\frac{R}{\bp{I}{n}})_{j+p^n})e^{-iy(j+p^n)/p^n}+ g_n(y) \ .
\end{split}    
\end{equation}
In \Cref{computation of limits of parts}, we compute limits of the terms appearing on the right hand side of \eqref{splitting into terms} normalized by $(\frac{1}{p^n})^2$.
\begin{lemma}\label{computation of limits of parts}For $y \in U$, we have
\begin{enumerate}
    \item $\underset{n \to \infty}{\lim}(\frac{1}{p^n})^2g_n(y)= 0.$ 
    \item $\underset{n \to \infty}{\lim}(\frac{1}{p^n})^2\underset{j < p^n}{\sum} \lambda((\frac{R}{\bp{I}{n}})_j)e^{-iyj/p^n}=-\delta_R(\frac{e^{-iy}-1}{(iy)^2})-\delta_R \frac{e^{-iy}}{iy} .$
    \item
    $\underset{n \to \infty}{\lim} (\frac{1}{p^n})^2\sum \limits_{j=2g-1}^{l_1(n)-1} \lambda((\frac{R}{\bp{I}{n}})_{j+p^n})e^{-iy(j+p^n)/p^n}.$\\
    $= - (\sum \limits_{b=0}^{t} \mu_{b+1}r_{b+1})(\frac{1-e^{
    \frac{iy\mu_1}{\delta_R}}}{iy})e^{-iy}
    -i\delta_R(\sum \limits_{b=0}^{t}r_{b+1})\dfrac{d}{dy}(\frac{1-e^{\frac{iy\mu_1}{\delta_R}}}{iy})e^{-iy}.$
    \item For any $s$, $1 \leq s \leq t$, $\underset{n \to \infty}{\lim}(\frac{1}{p^n})^2\sum \limits_{j= u_s(n)}^{l_{s+1}(n)-1} \lambda((\frac{R}{\bp{I}{n}})_{j+p^n})e^{-iy(j+p^n)/p^n}.$\\
    $= (-\sum \limits_{b=s}^{t} \mu_{b+1}r_{b+1})(\frac{e^{\frac{iy\mu_s}{\delta_R}}-e^{\frac{iy\mu_{s+1}}{\delta_R}}}{iy}) e^{-iy}
    -i\delta_R(\sum \limits_{b=s}^{t} r_{b+1}) \dfrac{d}{dy}(\frac{e^{\frac{iy\mu_s}{\delta_R}}-e^{\frac{iy\mu_{s+1}}{\delta_R}}}{iy})e^{-iy}.$
    
\end{enumerate}
\end{lemma}
\vspace{.5cm}
\noindent \textbf{\textit{Continuation of proof of \Cref{statement in the curve case}:}} We establish the statement \Cref{statement in the curve case}
using \Cref{computation of limits of parts} before verifying \Cref{computation of limits of parts}. When we use \Cref{computation of limits of parts} to compute $\underset{n \to \infty}{\lim}(\frac{1}{p^n})^2\underset{j \in \N}{\sum}\lambda((\frac{R}{\bp{I}{n}})_j)e^{-iyj/p^n}$, some terms on the right hand side of 3., \Cref{computation of limits of parts} cancel some terms on the right hand side of 4., \Cref{computation of limits of parts}. After cancelling appropriate terms we get, for $y \in U$
\begin{center}
$\underset{n \to \infty}{\lim}(\frac{1}{p^n})^2\underset{j \in \N}{\sum}\lambda((\frac{R}{\bp{I}{n}})_j)e^{-iyj/p^n}$\\
$= -\delta_R(\frac{e^{-iy}-1}{(iy)^2})-\delta_R \frac{e^{-iy}}{iy}- \frac{(\sum \limits_{b=0}^{t} \mu_{b+1}r_{b+1})e^{-iy}}{iy}+ \frac{\sum \limits_{b=0}^{t}\mu_{b+1}r_{b+1}e^{-iy(1- \frac{\mu_{b+1}}{\delta_R})}}{iy}$\\
$-i\delta_R(\sum \limits_{b=0}^{t}r_{b+1})\frac{d}{dy}(\frac{1}{iy})e^{-iy}+i \delta_R \sum \limits _{b=0}^{t}r_{b+1}\frac{d}{dy}(\frac{e^{iy\frac{\mu_{b+1}}{\delta_R}}}{iy})e^{-iy}$\\
$= -\delta_R(\frac{e^{-iy}-1}{(iy)^2})-\delta_R \frac{e^{-iy}}{iy}+ \frac{\delta_Re^{-iy}}{iy}+ \frac{\sum \limits_{b=0}^{t}\mu_{b+1}r_{b+1}e^{-iy(1- \frac{\mu_{b+1}}{\delta_R})}}{iy}$\\
$-\delta_R(\text{rk}(\mathcal{S}))\frac{e^{-iy}}{(iy)^2}- \frac{\sum \limits_{b=0}^{t}\mu_{b+1}r_{b+1}e^{-iy(1-\frac{\mu_{b+1}}{\delta_R})}}{iy}+ \delta_R\frac{\sum \limits_{b=0}^{t}r_{b+1}e^{-iy(1-\frac{\mu_{b+1}}{\delta_R})}}{(iy)^2}$ \ .
\end{center}
The last line is indeed equal to the right hand side of \eqref{curve}.
\begin{proof}[\textbf{Proof of \Cref{computation of limits of parts}:}]
1) We show that there is a constant $C$ such that $|g_n(y)| \leq Cp^n$ on $U$. By \Cref{length in terms of cohomologies lemma} and \Cref{cohomologies in terms of HN filtration}, $\lambda((\frac{R}{\bp{I}{n}})_l)=0$ for $l > -\mu_{t+1}\frac{p^n}{\delta_R}+ p^n+ 2g-2$. So using \eqref{splitting into terms} we get an integer $N$, such that each $g_n$ is a sum of at most $N$ functions of the form $\lambda((\frac{R}{\bp{I}{n}})_l)e^{-iyl/p^n}$, where $l$ is at most $-\mu_{t+1}\frac{p^n}{\delta_R}+p^n+2g-2$. We prove (1) by showing that there is a $C'$ such that for each of these functions $\lambda((\frac{R}{\bp{I}{n}})_l)e^{-iyl/p^n}$ appearing in $g_n$, $|\lambda((\frac{R}{\bp{I}{n}})_l)e^{-iyl/p^n}|\leq C'p^n$ on $U$. For that, note since $U$ has a compact closure, there is a constant $C_1$ such that, for all $l  \leq -\mu_{t+1}\frac{p^n}{\delta_R}+p^n+2g-2$,  $|e^{-iyl/p^n}|$ is bounded above by $C_1$ on $U$. Since $\lambda((\frac{R}{\bp{I}{n}})_l) \leq \lambda(R_l)$ and there is a constant $C_2$ such that for all $l$, $\lambda(R_l) \leq C_2. l$, we are done.\\\\
(2)
\begin{equation}\label{the first term}
\begin{split}
\underset{n \to \infty}{\lim}(\frac{1}{p^n})^2\underset{j < p^n}{\sum} \lambda((\frac{R}{\bp{I}{n}})_j)e^{-iyj/p^n} &= \underset{n \to \infty}{\lim}(\frac{1}{p^n})^2\underset{j < p^n}{\sum} \lambda(R_j)e^{-iyj/p^n}\\
&= \underset{n \to \infty}{\lim}(\frac{1}{p^n})^2 \sum \limits_{j=0}^{p^n-1} \lambda(R_j)e^{-iyj/p^n}\frac{p^n(1-e^{-iy/p^n})}{iy}\\
&= \underset{n \to \infty}{\lim}\frac{1}{p^n}\sum \limits_{j=0}^{p^n-1} \lambda(R_j) \int_{j/p^n}^{(j+1)/p^n}e^{-iyx}dx.
\end{split}
\end{equation}
Now consider $(h_n)_n$ the sequence of real valued functions on $[0,1]$ defined by $h_n(x)= \frac{1}{p^n}\lambda(R_{\lfloor xp^n \rfloor})$. The last line of \eqref{the first term} is then $\int_{0}^{1}h_n(x)e^{-iyx}dx$. Since $h_n(x)$ converges to the function $\delta_R x$ uniformly on $[0,1]$, we have
\begin{equation}\label{final form of the first term}
\begin{split}
\underset{n \to \infty}{\lim}(\frac{1}{p^n})^2\underset{j < p^n}{\sum} \lambda((\frac{R}{\bp{I}{n}})_j)e^{-iyj/p^n} &= \delta_R\int_{0}^{1} xe^{-iyx}dx\\
&=-\delta_R(\frac{e^{-iy}-1}{(iy)^2})-\delta_R \frac{e^{-iy}}{iy} \ .
\end{split}
\end{equation}
(3) Using \Cref{length in terms of cohomologies lemma} and \Cref{cohomologies in terms of HN filtration}, we get,
\begin{equation}\label{the second term}
\begin{split}
  &(\frac{1}{p^n})^2\sum \limits_{j=2g-1}^{l_1(n)-1} \lambda((\frac{R}{\bp{I}{n}})_{j+p^n})e^{-iy(j+p^n)/p^n}\\
  & = [-\frac{e^{-iy}}{p^n} (\sum \limits_{b=0}^{t} \mu_{b+1}r_{b+1})+e^{-iy}(\frac{1}{p^n})^2 (\sum \limits_{b=0}^{t} r_{b+1})(g-1)](\sum \limits_{j=2g-1}^{l_1(n)-1} e^{-iyj/p^n})\\
  &-e^{-iy}(\frac{1}{p^n})^2(\sum \limits_{b=0}^{t} r_{b+1})\delta_R\sum \limits_{j=2g-1}^{l_1(n)-1} j e^{-iyj/p^n}\\
  &= [-\frac{1}{p^n}\sum \limits_{b=0}^{t} \mu_{b+1}r_{b+1}+ (\frac{1}{p^n})^2(\sum \limits_{b=0}^{t} r_{b+1})(g-1)]\frac{1-e^{-iy(l_1(n)-2g+1)/p^n}}{1- e^{-iy/p^n}} e^{-iy(1+\frac{2g-1}{p^n})}\\
  & -\frac{i}{p^n}(\sum \limits_{b=0}^{t} r_{b+1})\delta_R \dfrac{d}{dy}(\frac{e^{-iy(2g-1)/p^n}- e^{-iy l_1(n)/p^n}}{1-e^{-iy/p^n}})e^{-iy} \,\, \, \, (\text{using} \ \frac{d}{dy}(e^{-iyj/p^n})= \frac{-ij}{p^n}e^{-iyj/p^n}) \ .
\end{split} 
\end{equation}
While taking limit as $n$ approaches infinity of the last line in \eqref{the second term}, the order of differentiation and taking limit can be exchanged (see \cite{Ahlfors}, Chapter 5, Theorem 1). So taking limit as $n$ approaches infinity in \eqref{the second term} and using \eqref{two limits}, \Cref{the essential limit} we get (3).\\\\
(4) \Cref{length in terms of cohomologies lemma}, \Cref{cohomologies in terms of HN filtration} and a computation as in the proof of (3) shows that for $1 \leq s \leq t$,
\begin{equation}\label{the third term}
\begin{split}
&\sum \limits_{j= u_s(n)}^{l_{s+1}(n)-1} \lambda((\frac{R}{\bp{I}{n}})_{j+p^n})e^{-iy(j+p^n)/p^n}\\
& = [-\frac{1}{p^n}\sum \limits_{b=s}^{t} \mu_{b+1}r_{b+1}+ (\frac{1}{p^n})^2(\sum \limits_{b=s}^{t} r_{b+1})(g-1)]\frac{e^{-iyu_s(n)/p^n}-e^{-iyl_{s+1}(n)/p^n}}{1- e^{-iy/p^n}} e^{-iy}\\
& -\frac{i}{p^n}(\sum \limits_{b=s}^{t} r_{b+1})\delta_R \dfrac{d}{dy}(\frac{e^{-iyu_s(n)/p^n}-e^{-iyl_{s+1}(n)/p^n}}{1- e^{-iy/p^n}})e^{-iy} \ .
\end{split}
\end{equation}
Now (4) follows from taking limit as $n$ approaches infinity and arguing as in the proof of (3).
\end{proof}
\end{proof}


\printbibliography[title= {References}]

\end{document}